%% file: main.tex
\documentclass[12pt,a4paper]{amsart}
\usepackage[foot]{amsaddr}

\makeatletter
\input{commands.tex}
\input{symbols.tex}
\makeatother

\bibliography{bibliography}

\begin{document}
	\title[Breather solutions for semilinear wave equations]{Breather solutions for semilinear wave equations}

	\author{Julia Henninger$^1$}
	\email{julia.henninger@kit.edu}
	\author{Sebastian Ohrem$^1$}
	\email{sebastian.ohrem@kit.edu}
	\author{Wolfgang Reichel$^1$}
	\email{wolfgang.reichel@kit.edu}
	\address{$^1$Institute for Analysis, Karlsruhe Institute of Technology (KIT), D-76128 Karlsruhe, Germany}
	
	\date{\today}
	\subjclass[2020]{Primary: 35L71, 49J35; Secondary: 35B10, 34L05}
	\keywords{semilinear wave equation, breather solutions, time-periodic solutions, variational methods, functional calculus, spectral measure}

	\begin{abstract}
		\input{abstract.tex}
	\end{abstract}

	\maketitle

	\input{introduction.tex}

	\input{proof_main_results.tex}
	
	\input{functional_calc.tex}

	\appendix

	\input{vector_L2.tex}

	\input{eigenfunction_bounds.tex}

    \input{examples}

    \section*{Acknowledgments}
	Funded by the Deutsche Forschungsgemeinschaft (DFG, German Research Foundation) – Project-ID 258734477 – SFB 1173. We thank Xian Liao, Michael Plum, and Robert Wegner for discussions on Floquet-Bloch theory and scattering which helped in the development of \cref{sec:embeddings}.

	\printbibliography
\end{document}

%% file: commands.tex
\usepackage[english]{babel}
\usepackage[utf8]{inputenc}
\usepackage{csquotes}
\usepackage[T1]{fontenc}
\usepackage{latexsym}
\usepackage[leqno]{amsmath}
\usepackage{amssymb,amsthm,amsfonts}
\usepackage{mathtools}
\usepackage{mathrsfs}
\usepackage{dsfont}
\usepackage{stmaryrd}

\usepackage{todonotes}

\usepackage{geometry}
\usepackage{nicefrac}

\usepackage{enumerate}
\usepackage[shortlabels]{enumitem}
\setlist[enumerate]{label=(\alph*),font=\normalshape}
\setlist[itemize]{font=\normalshape}
\let\originalitem\item
\renewcommand{\item}[1][]{%
	\if\relax\detokenize{#1}\relax%
		\originalitem%
	\else%
		\originalitem[#1]%
		\phantomsection
		\def\@currentlabel{#1}
	\fi%
}

\usepackage{hyperref}

\usepackage[noabbrev,capitalize]{cleveref}

\usepackage[style=numeric-comp]{biblatex}

\let\originalleft\left
\let\originalright\right
\renewcommand{\left}{\mathopen{}\mathclose\bgroup\originalleft}
\renewcommand{\right}{\aftergroup\egroup\originalright}

\allowdisplaybreaks
\theoremstyle{plain}
\newtheorem{theorem}{Theorem}[section]
\newtheorem{definition}[theorem]{Definition}
\newtheorem{proposition}[theorem]{Proposition} 
\newtheorem{lemma}[theorem]{Lemma} 

\newtheorem{notation}[theorem]{Notation}

\theoremstyle{definition}
\newtheorem{remark}[theorem]{Remark}

\AddToHook{env/definition/begin}{\crefalias{theorem}{definition}}
\AddToHook{env/proposition/begin}{\crefalias{theorem}{proposition}}
\AddToHook{env/lemma/begin}{\crefalias{theorem}{lemma}}
\AddToHook{env/corollary/begin}{\crefalias{theorem}{corollary}}
\AddToHook{env/notation/begin}{\crefalias{theorem}{notation}}
\AddToHook{env/remark/begin}{\crefalias{theorem}{remark}}

\newcommand{\C}{\mathbb{C}} 
\newcommand{\R}{\mathbb{R}} 
 
\newcommand{\Z}{\mathbb{Z}} 
\newcommand{\Zodd}{\mathbb{Z}_\mathrm{odd}}

\newcommand{\N}{\mathbb{N}} 
\newcommand{\Nodd}{\mathbb{N}_\mathrm{odd}}
\newcommand{\Neven}{\mathbb{N}_\mathrm{even}}
\newcommand{\T}{\mathbb{T}}

\newcommand{\id}{I}

\DeclareMathAlphabet{\othermathbb}{U}{bbold}{m}{n}
\newcommand{\bbone}{\othermathbb{1}}

\newcommand{\der}[2][]{\@ifnextchar\der{\,#1\mathrm{d}#2\!}{\,#1\mathrm{d}#2}}
\newcommand{\Der}{\mathrm{d}}
\DeclareMathOperator{\dist}{dist}
\DeclareMathOperator{\supp}{supp}
\DeclareMathOperator{\lspan}{span}

\DeclareMathOperator*{\essinf}{ess~inf}
\DeclareMathOperator{\tr}{tr}
\DeclareMathOperator{\Id}{Id}
\renewcommand{\Re}{\operatorname{Re}}
\renewcommand{\Im}{\operatorname{Im}}

\newcommand{\landauO}{\mathcal{O}}
\newcommand{\landauo}{\hbox{o}}
\newcommand{\ee}{\mathrm{e}}
\newcommand{\ii}{\mathrm{i}}
\let\eps\varepsilon
\let\embeds\hookrightarrow

\newcommand{\set}[1]{{\left\{ #1 \right\}}}
\newcommand{\Set}[1]{\bigl\{ #1 \bigr\}}
\newcommand{\abs}[1]{\left\lvert #1 \right\rvert}
\newcommand{\Abs}[1]{\bigl\lvert #1 \bigr\rvert}
\newcommand{\norm}[1]{\left\lVert #1 \right\rVert}
\newcommand{\floor}[1]{\left\lfloor #1 \right\rfloor}
\newcommand{\ceil}[1]{\left\lceil #1 \right\rceil}

\newcommand{\impvar}{\,\cdot\,}
\newcommand{\ip}[2]{\left\langle #1 , #2 \right\rangle}
\newcommand{\dv}[3][]{%
	\if\relax\detokenize{#1}\relax%
		\frac{\mathrm{d} #2}{\mathrm{d} #3}%
	\else%
		\frac{\mathrm{d}^{#1} #2}{\mathrm{d} {#3}^{#1}}%
	\fi%
}
\newcommand{\dx}{\dv{}{x}}
\newcommand{\dxsquare}{\dv[2]{}{x}}

%% file: symbols.tex
\newcommand{\loc}{\mathrm{loc}}
\newcommand{\per}{\mathrm{per}}
\newcommand{\ess}{\mathrm{ess}}
\DeclareMathOperator*{\Var}{Var}
\DeclareMathOperator{\lcm}{lcm}

\newcommand{\calH}{\mathcal{H}}
\newcommand{\calL}{\mathcal{L}}
\newcommand{\calM}{\mathcal{M}}

\renewcommand{\H}{\mathbb{H}}

\renewcommand{\S}{\mathbb{S}}
\newcommand{\SwR}{\mathbb{S}_*}

%% file: abstract.tex
We prove existence of real-valued, time-periodic and spatially localized solutions (breathers) of semilinear wave equations $V(x)u_{tt} - u_{xx} = \Gamma(x) |u|^{p-1} u$ on $\R^2$ for all values of $p\in (1,\infty)$. Using tools from the calculus of variations our main result provides breathers as ground states of an indefinite functional under suitable conditions on $V, \Gamma$ beyond the limitations of pure $x$-periodicity. Such an approach requires a detailed analysis of the wave operator acting on time-periodic functions. Hence a generalization of the Floquet-Bloch theory for periodic Sturm-Liouville operators is needed which applies to perturbed periodic operators.  For this purpose we develop a suitable functional calculus for the weighted operator $-\frac{1}{V(x)}\frac{\mathrm{d}^2}{\mathrm{d}x^2}$ with an explicit control of its spectral measure. Based on this we prove embedding theorems from the form domain of the wave operator into $L^q$-spaces, which is key to controlling nonlinearities. We complement our existence theory with  explicit examples of coefficient functions $V$ and temporal periods $T$ which support breathers. 
 

%% file: introduction.tex

\section{Introduction and main results}

We are interested in breather solutions for the spatially heterogeneous semilinear wave equation 
\begin{equation}
    V(x)u_{tt} - u_{xx} = \Gamma(x) |u|^{p-1} u \mbox{ for } (x,t)\in\R^2 \label{eq:basic}
\end{equation}
for $p>1$. Breathers are real-valued, periodic in $t$ and localized in $x$, i.e., $\lim_{|x|\to\infty} u(x,t)=0$. Our goal is to give sufficient conditions on $V$ and $\Gamma$ so that non-trivial breathers of \eqref{eq:basic} exist. 
 
Breather solutions to \eqref{eq:basic} are a rare phenomenon in the sense that only few examples of coefficient classes $V(x)$ are known which support breathers -- a more detailed account of the literature will be given at the end of the introduction. 
The results known to us so far \cite{BlaSchneiChiril, hirschreichel, Maier, maier_reichel_schneider} and very recently \cite{martina_panos} all rely on spatially periodic coefficients $V_\per$ and on having very good information on the spectrum of the spatial operator $L=-\frac{1}{V_\per(x)}\dxsquare$. The main tool to describe periodic differential operators is Floquet-Bloch theory, which provides an explicit description of spectrum and spectral measure in terms of quasiperiodic eigenvalues and eigenfunctions on the periodicity cell. The basic idea in the approaches cited above is then to tailor $V_\per$ in such a way that $L$ has spectral gaps about $k^2 \omega^2$ for $k\in \Nodd$, with $\omega$ being the temporal frequency of the breather, e.g. by a careful design of a piecewise constant $V_\per$. As a consequence, the linear wave operator $V(x)\partial_t^2-\partial_x^2$ is invertible in suitable spaces of $T=\frac{2\pi}{\omega}$-periodic functions. 
In principle, this approach is not limited to spatially periodic coefficients $V_\per(x)$, but as soon as one leaves this class, several difficulties arise: 
\begin{itemize}
\item[(a)] finding a way to describe the spectrum of $L$ and to keep $k^2\omega^2$ out of it,
\item[(b)] finding a replacement for the Bloch-transform which diagonalizes the operator $L$. 
\end{itemize} 
In this paper, we overcome the limitations of spatially periodic coefficients to some extent and show existence of breathers to \eqref{eq:basic} for different kinds of perturbed periodic $V$. A prototypical case where our results apply is a perturbation $V=V_\per+V_\loc$ of a positive periodic potential $V_\per$ where $V_\loc$ has compact support. Concrete examples beyond the purely periodic setting are also provided.

As an essential tool to find breathers in the nonperiodic setting and in extension of \cite{coddington}, we show an explicit formula for the spectral measure of the perturbed periodic operator $L = - \frac{1}{V(x)} \dxsquare$, cf. \cref{thm:spectral_measure:explicit}. The spectral description uses generalized eigenvalues and eigenfunctions replacing the quasiperiodic Bloch-eigenfunctions, and their interaction in the perturbed region. Once this is available we use the calculus of variations  to find breathers of \eqref{eq:basic} as critical points of a functional $J$. We show that $J$ is well defined in the Hilbert space $\mathcal{H}=\{u\in L^2(\R\times\T): \int_{\R\times\T} u|\Box|u V(x)\der[] (x,t)<\infty\}$ where $\Box\coloneqq \partial_t^2 - \frac{1}{V(x)} \partial_x^2$ is the weighted wave operator and $|\Box|$ its spectrally defined absolute value. To achieve this, we show embedding properties of $\mathcal{H}$ into $L^q(\R\times\T)$-spaces using the functional calculus for $L$. Then, saddle-point tools from the calculus of variations \cite{szulkin_weth} provide breathers for \eqref{eq:basic} as ground states of $J$. 

With this brief introduction we can now describe our main result. We consider potentials $V$ which coincide outside a finite interval with a periodic function $V_\per^+$ in a neighborhood of $+\infty$ and $V_\per^-$ in a neighborhood of $-\infty$. This also includes the case where $V$ is a purely periodic function. We use the following basic assumptions. To improve readability, we use the $\pm$ symbol, and statements involving double symbols should be read choosing always the top symbol, or always the bottom symbol. 
\begin{enumerate}
    \item[(A1)]
	$V, \Gamma \in L^\infty(\R)$ with $\essinf_\R V > 0$ and $\Gamma > 0$ almost everywhere. Moreover, $V$ has locally bounded variation.
	\label{ass:bounded} \label{ass:first}
    \item[(A2)] There exist $X^\pm>0$, $R^\pm\in \R$, and $X^\pm$-periodic functions $V_\per^\pm\in L^\infty(\R)$ such that $V(x)=V^+_\per(x)$ for $x>R^+$ and $V(x)=V^-_\per(x)$ for $x<R^-$.  \label{ass:periodic_infinity}
\end{enumerate}

The next two assumptions concern the spectrum of the weighted Sturm-Liouville operator $L$. If we consider the $\frac{1}{V(x)}$-weighted wave operator $\Box =\partial_t^2 - \frac{1}{V(x)} \partial_x^2$ applied to a time periodic function $u$, then a Fourier-decomposition of $u=\sum_{k\in \Z} \hat u_k(x) \ee^{\ii k\omega t}$ results into a Fourier-decomposition of $\Box$ into a family $(L_k)_{k\in \Z}$ of Sturm-Liouville operators
\begin{align}\label{eq:def:Lk}
	L_k\coloneqq-\frac{1}{V(x)}\dxsquare - k^2\omega^2,	
\end{align}
where $\omega=\frac{2\pi}{T}$ is the frequency and $T$ is the time period of the breather. As we shall see later, it is advantageous to consider only $k\in\Zodd$ which amounts to a restriction to $\frac{T}{2}$-antiperiodic functions $u$. 
Since $L_k = L_{-k}$, the main assumption is now to keep $0$ out of the spectrum of the family $(L_k)_{k\in\Nodd}$ or, in other words, to keep $k^2\omega^2$ out of the spectrum $\sigma(L)$ of the weighted Sturm-Liouville operator $L = -\frac{1}{V(x)}\dxsquare$. In fact, a stronger assumption is needed as follows.

\begin{enumerate}
    \item[(A3)] There exists $\omega>0$ such that $\inf\{|\sqrt{\lambda} - k\omega|: \lambda\in \sigma(L), k\in \Nodd\}>0$. 
	\label{ass:spectrum}

	\item[(A4)] 
	The point spectrum $\sigma_p(L)$ satisfies $\sum_{\lambda \in \sigma_p(L)} \lambda^{-r} < \infty$ for all $r > \frac12$, which means that $\sigma_p(L)$ grows at least quadratically or is finite.
	\label{ass:point_spectrum}
	\label{ass:last}
\end{enumerate}

The restriction to $k\in\Nodd$ in \ref{ass:spectrum}, or in other words the restriction to $\frac{T}{2}$-antiperiodic functions, avoids the $k=0$ mode in the Fourier decomposition which would otherwise be incompatible with \ref{ass:spectrum}. 
With these assumptions we can now formulate our main theorem. We use the notation $\T \coloneqq \R \slash_{T \Z}$ for the one-dimensional torus of length $T$.

\begin{theorem} \label{thm:main} Let $1<p<\infty$ and assume \ref{ass:bounded}, \ref{ass:periodic_infinity}, \ref{ass:spectrum}, \ref{ass:point_spectrum}. Then \eqref{eq:basic} has infinitely many nontrivial $T=\frac{2\pi}{\omega}$-periodic breathers if additionally one of the following assumptions are satisfied.
\begin{enumerate}
    \item[(a)] (compact case) $\lim_{|x|\to \infty} \Gamma(x)=0$.
    \item[(b)] (asymptotically periodic case) $V=V_\per$ is $X$-periodic on $\R$ and $\Gamma=\Gamma_\per+\Gamma_\loc$ where $\Gamma_\per$ is $X$-periodic and $\Gamma_\loc\geq 0$, $\lim_{|x|\to\infty} \Gamma_\loc(x)=0$.
\end{enumerate}
The solutions are strong $H^2(\R\times\T)$-solutions and satisfy \eqref{eq:basic} pointwise almost everywhere. 
The theorem also holds if we replace $\Gamma$ by $-\Gamma$ in \eqref{eq:basic}.
\end{theorem}

\begin{remark}
    Note that in case (b) we have $V_\per^+=V_\per^-=V_\per$ and $X^+=X^-=X$. Case (b) also includes the purely periodic case where $\Gamma_\loc=0$. Assumption (a) provides a convenient way to overcome compactness problems for Palais-Smale sequences. It remains open how to generalize (a), (b) to cases where $\inf_\R\Gamma>0$ and still $V$ is a perturbation of a periodic potential. 
\end{remark}

Examples of coefficients $V$ which satisfy \ref{ass:bounded}, \ref{ass:periodic_infinity}, \ref{ass:spectrum}, \ref{ass:point_spectrum} are given in Appendix~\ref{sec:examples}. The structure of the paper will be given at the end of the introduction after the account of the relevant literature which follows next.

The study of breather solutions to nonlinear wave equations may have its origins in completely integrable cases such as the Sine-Gordon equation $u_{tt}-u_{xx}+\sin u=0$ on $\R\times\R$ where explicit breather families can be constructed through the inverse scattering method \cite{ablowitz}. Although next to Sine-Gordon, a number of completely integrable systems with breather solutions are known, e.g., Korteweg-de Vries, nonlinear Schr\"odinger, Toda lattice etc., these systems are still rare, and therefore the search for methods to prove existence of breathers beyond inverse scattering continues. The first attempts of studying generalizations of the Sine-Gordon equation turned out to be fruitless since in \cite{birnir_etal, denzler, segur_kruskal, kowalczyk_et_al} it was shown that if the $\sin u$ nonlinearity in Sine-Gordon is generalized to an analytic perturbation $f(u)$ with $f(0) = 0$, $f'(0)=1$, $f''(0)=0$ then only the $f(u)=\sin u$ nonlinearity can support a breather solution. On the other hand, nonlinear lattice equations can support breather solutions \cite{mackay_aubry}, e.g. the 2-atomic Fermi-Pasta-Ulam-Tsingou chains \cite{james_noble, james_cuevas}. A similar result on the PDE-side of nonlinear wave equations was obtained by Blank, Chirilus-Bruckner, Lescarret and Schneider \cite{BlaSchneiChiril}, where breathers for spatially-periodic nonlinear wave equations of the type \eqref{eq:basic} were shown to exist. As we have explained above, the idea was to use the band-structure of the spectrum of the spatially periodic operator $L=-\frac{1}{V(x)}\dxsquare$ and design the coefficient function $V(x)$ in such a way that the temporal frequencies $k^2\omega^2$ with $\omega=\frac{2\pi}{T}$ being the breather frequency, fall into the gaps of the spectrum of $L$,  cf. \cite{bruckner_wayne} for a systematic approach to such coefficient functions based on inverse spectral theory. By adding a suitable linear potential with a bifurcation parameter to $L$, breather solutions bifurcating from the edge of the essential spectrum were constructed by the center-manifold theorem and spatial dynamics. The same approach was also used for finding breathers for constant-coefficient nonlinear wave equations on metric graphs \cite{Maier}, where the spatial heterogeneity of the graph generates a banded spectrum of the spatial operator.

A methodically different approach using calculus of variation tools rather than spatial dynamics came up in \cite{hirschreichel} and was also used in \cite{maier_reichel_schneider}. Variational methods for time-periodic solutions of nonlinear wave equations on spatially bounded intervals have been used before  \cite{brezis_coron_nirenberg, brezis_coron, hofer}. But its applicability to breathers on unbounded spatial domains like the real line has been substantially obstructed due to essential spectra rather than point spectra of the underlying spatial differential operators. The variational techniques constructed breathers as saddle points of an energy functional on a suitable Hilbert space.
It was successful for \eqref{eq:basic} with a periodic potential $V(x)$ \cite{hirschreichel} or on a periodic metric graph \cite{maier_reichel_schneider}, and it required essentially the same spectral situation as in the spatial dynamics approaches. As an additional feature the variational method does not rely on a bifurcation parameter and can therefore generate "large" breathers instead of "small" breathers locally bifurcating from zero as for the spatial dynamics method.  

From the results mentioned so far one might get the impression that (except for the completely integrable cases) spatial heterogeneity is a necessary prerequisite for the existence of breathers. It is surprising that in space dimensions 2 or higher this is not the case, cf. \cite{mandel_scheider_breather, scheider_breather} for the construction of weakly localized breathers.

Our paper is structured as follows: in Section~\ref{sec:proof_main} we set up the problem in a way that it can be treated by tools from variational calculus, i.e., we give a proper definition of the energy functional $J$ and its domain $\mathcal{H}$. Then we prove that ground states exists and are strong solutions of \eqref{eq:basic}. In Section~\ref{sec:embeddings} the function calculus for $L$ is introduced and used to prove embedding theorems $\mathcal{H}\embeds L^q(\R\times\T)$. Three appendices complete our paper: in Appendix~\ref{sec:vector_L2} we collect some definitions and properties of vector valued $L^2$-spaces based on positive-definite matrix-valued measures, which appear in the functional calculus for $L$. In Appendix~\ref{sec:eigenfunction_bounds} we show how $L^2$ and $L^\infty$-bounds for solutions of $-u''=\lambda V(x)u$ can be mutually estimated uniformly in $\lambda$. Finally, in Section~\ref{sec:examples} we give several classes of potential $V$ for which the assumptions 
\ref{ass:bounded}, \ref{ass:periodic_infinity}, \ref{ass:spectrum}, \ref{ass:point_spectrum} our our main result of Theorem~\ref{thm:main} hold.

%% file: proof_main_results.tex

\section{Proof of the main result} \label{sec:proof_main}

We begin with the functional analytic framework. The one-dimensional torus $\T = \R \slash_{T \Z}$ is equipped with the measure $\Der t = \frac{1}{T} \Der \lambda$ where $\Der\lambda$ is the Lebesgue measure on $[0, T]$. We use the notation $L^2_V(\R) \coloneqq L^2(\R, V(x)\Der x)$ and $L^2_V(\R\times\T) \coloneqq L^2(\R\times\T,V(x)\der (x,t))$. 
Our goal is to obtain breathers as critical points of a suitable functional. Purely formally, this can be achieved via the functional
$$
J(u) = \int_{\R\times\T} -V(x)u_t^2 + u_x^2 - \frac{2}{p+1} \Gamma(x) |u|^{p+1} \der[] (x,t)
$$
where $u\colon\R\times\R\to\R$ is a temporally $\frac{T}{2}$-antiperiodic function of time. 
Note that we have not yet specified the domain of the functional $J$. This will be our next task. First, we decompose $u$ into its temporal Fourier modes by writing 
$$
u(x,t) = \sum_{k\in \Zodd} \hat u_k(x) e_k(t)
$$
where $e_k(t) = \ee^{\ii \omega k t}$, $k\in\Z$ is an orthonormal basis of $L^2(\T)$. Then the functional $J$ takes the form 
$$
J(u) = J_0(u) - J_1(u)
$$
where 
$$
J_0(u) = \sum_{k\in \Zodd}\int_\R |\hat u_k'|^2 -k^2\omega^2 V(x)|\hat u_k|^2 \der x, 
\quad
J_1(u) = \tfrac{2}{p+1}\int_{\R\times\T} \Gamma(x) |u|^{p+1} \der (x,t).
$$
For the operator $L_k$ as defined in \eqref{eq:def:Lk}, we know by \ref{ass:spectrum} that $0$ does not belong to its spectrum.
Therefore the following constructions are possible.
Based on the spectral resolution $L=\int_\R \lambda \der P_\lambda$ of the selfadjoint operator $L: H^2(\R)\subset L^2_V(\R)\to L^2_V(\R)$ we can define 
$$
\sqrt{|L_k|} = \int_\R \sqrt{|\lambda-k^2\omega^2|} \der P_\lambda,
$$
which has domain $H^1(\R)$. 
We equip $H^1(\R)$ with the norm $\|v\|_{\mathcal{H}_k}\coloneqq \|\sqrt{|L_k|}v\|_{L^2_V}$ and thus obtain the Hilbert space $\mathcal{H}_k$.
The two orthogonal projections $P_k^\pm \colon \calH_k \to \calH_k$ given by
\begin{align*}
	v^+ \coloneqq P_k^+[v] \coloneqq \int_{k^2\omega^2}^\infty 1 \der P_\lambda[v],
	\qquad
	v^- \coloneqq P_k^-[v] \coloneqq \int_{-\infty}^{k^2\omega^2} 1 \der P_\lambda[v]
\end{align*}
yield an orthogonal decomposition $\mathcal{H}_k=\mathcal{H}_k^+\oplus\mathcal{H}_k^-$, $v = v^+ + v^-$ such that
\begin{align*}
	\int_\R |v'|^2 -k^2\omega^2 V(x)|v|^2 \der x =\|v^+\|_{\mathcal{H}_k}^2-\|v^-\|_{\mathcal{H}_k}^2, \quad \|v\|_{\mathcal{H}_k}^2 =\|v^+\|_{\mathcal{H}_k}^2+\|v^-\|_{\mathcal{H}_k}^2.
\end{align*}
For later purposes, let us also introduce the bilinear form associated to $L_k$ by 
$$
b_k(v,w) = \int_\R v' w' - k^2\omega^2 V(x) vw\der x, \quad v,w \in \mathcal{H}_k.
$$
 
\medskip

As we shall see in \cref{prop:Lp-embedding}, the proper domain for $J$ is the Hilbert space 
$$
\mathcal{H} \coloneqq \{u\in L^2(\R\times\T): u(x,t) = \sum_{k\in \Zodd} \hat u_k(x) e_k(t), \;\hat u_k=\overline{\hat u_{-k}} \;\forall k\in\Zodd, \;\|u\|_\mathcal{H}<\infty\}
$$
equipped with the norm $\|\cdot\|_\mathcal{H}$ according to 
$$
\|u\|_\mathcal{H}^2 \coloneqq \sum_{k\in \Zodd} \langle \sqrt{|L_k|}\hat u_k,\sqrt{|L_k|}\hat u_k\rangle_{L^2_V}=2\sum_{k\in \Nodd} \norm{\hat u_k}^2_{\mathcal{H}_k}.
$$
By \ref{ass:spectrum} we have that $\mathcal{H}$ continuously embeds into $L^2(\R\times\T)$.
Moreover, $\mathcal H = \oplus_{k\in \Nodd} \mathcal{H}_k$ and the orthogonal decomposition of $\mathcal{H}_k=\mathcal{H}_k^+\oplus \mathcal{H}_k^-$ now readily extends to $\mathcal{H}$ so that $\mathcal{H}=\mathcal{H}^+\oplus\mathcal{H}^-$, $\mathcal H^\pm = \oplus_{k\in \Nodd} \mathcal{H}_k^\pm$ and
\begin{align*}
J_0(u) =\|u^+\|_{\mathcal{H}}^2-\|u^-\|_{\mathcal{H}}^2, \quad \|u\|_{\mathcal{H}}^2 =\|u^+\|_{\mathcal{H}}^2+\|u^-\|_{\mathcal{H}}^2
\end{align*}
where $u=u^++u^-$ and $u^\pm=P^\pm[u]\in \mathcal{H}^\pm$ according to the two orthogonal projectors $P^\pm : \mathcal{H} \to \mathcal{H}^\pm$.

\begin{remark} \label{rem:dense}
    It is clear that functions $u=\sum_{|k|\leq K} \hat u_k e_k(t)$ with finitely many non-zero Fourier-modes and where $\hat{u}_k\in H^1(\R)$ has compact support are dense in $\mathcal{H}$.
\end{remark}

Since $J$ is neither bounded from above nor from below, a critical point of $J$ is necessarily a saddle point. A general tool for obtaining existence of saddle points is given by the generalized Nehari-Pankov manifold, cf. \cite{pankov, szulkin_weth}. It consists of minimizing $J$ on the set 
$$
\mathcal{M} = \{ u\in \mathcal{H}\setminus\mathcal{H}^-: J'(u)[w] = 0 \mbox{ for all }
w\in [u]+\mathcal{H}^-\}.
$$
The following properties of $J$ and $\mathcal{M}$ are important prerequisites for applying abstract results from \cite{szulkin_weth}.   

\begin{lemma} \label{lem:property_J_M} The following statements hold true:
\begin{itemize}
\item[(i)] $J_1$ is weakly lower semicontinuous, 
\begin{align*} 
J_1(0)=0 \quad \text{ and } \quad \frac{1}{2}J_1'(u)[u]> J_1(u)>0 \text{ for } u\neq 0.
\end{align*}
\item[(ii)] $\lim_{u\to 0} \frac{J_1'(u)}{\|u\|_{\mathcal{H}}}=0$ and $\lim_{u\to 0} \frac{J_1(u)}{\|u\|^2_{\mathcal{H}}}=0$.
\item[(iii)] For a weakly compact set $U\subset \mathcal{H}\setminus\{0\}$ we have $\lim_{s\to\infty} \frac{J_1(s u)}{s^2} = \infty$ uniformly w.r.t. $u\in U$.
\item[(iv)] In case (a) of Theorem~\ref{thm:main} the map $u \mapsto J_1'(u)$ is completely continuous from $\mathcal{H}$ to $\mathcal{H}'$.
\item[(v)] For each $w\in\mathcal{H}\setminus \mathcal{H}^-$ let $\mathcal{H}(w)=\R_{\geq 0}w+ \mathcal{H}^-$. Then there exists a unique nontrivial critical point $m(w)$ of $J|_{\mathcal{H}(w)}$. Moreover, $m(w)\in \mathcal{M}$ is the unique global maximizer of $J|_{\mathcal{H}(w)}$ as well as $J(m(w))>0$.
\item[(vi)] There exists $\delta>0$ such that $\| m(w)^+\|_{\mathcal{H}} \geq \delta$ for all $w\in \mathcal{H}\setminus \mathcal{H}^-$.
\end{itemize}
\end{lemma}

\begin{proof}
The proof of (iv) is standard. The proof of the remaining claims is essentially contained in \cite{maier_reichel_schneider}. For (i)--(iii) see the proof of \cite[Lemma~5.1]{maier_reichel_schneider} where for Fatou's lemma the positivity of $\Gamma$ is used, cf. \ref{ass:first}. For (v)--(vi) we refer to the proof of \cite[Lemma~5.2]{maier_reichel_schneider}.
\end{proof}

Another important property of $\mathcal{M}$ is that is does not generate Lagrange-multipliers for critical points of $J$ when restricted to $\mathcal{M}$. This is the main meaning of the following result whose proof can be found in \cite{bartsch_dohnal_plum_reichel}.
 \begin{lemma}
     The set $\mathcal{M}$ is a $C^1$-manifold, Moreover, if $\mathcal{M}_0$ is a bounded subset of $\mathcal{M}$ then there exists a constant $C>0$ with the following property: if $u\in \mathcal{M}_0$ , $\nabla J[u]=\tau+\sigma$ with $\tau\in T_u\mathcal{M}$ and $\sigma\perp_\mathcal{H} \tau$ then 
     $$
     \|\nabla J(u)\|_{\mathcal H}\leq C\|\tau\|_{\mathcal H}.
     $$
 \end{lemma}

 The next result provides the important boundedness of Palais-Smale sequences. It differs methodically in its proof from \cite{maier_reichel_schneider}. In the following proofs $C$ denotes a constant that is independent of $u$ but may change from line to line.

\begin{lemma} \label{lem:boundedness_of_PS}
There exist constants $C, \eps>0$ such that 
\begin{align}\label{eq:loc:manifold_est}
	\eps\leq \|u\|_\mathcal{H} \leq C J(u)^\frac{p}{p+1} \mbox{ for all } u\in \mathcal{M}.
\end{align}
Moreover, any Palais-Smale sequence $(u_n)_{n\in\N}$ of $J:\mathcal{H}\to \R$ is bounded.
\end{lemma}

\begin{proof} 
	In the following we use the embedding of $\mathcal{H}$ into $L^{p+1}(\R\times\T)$, cf. \cref{prop:Lp-embedding}. For $u\in \mathcal{M}$ we find 
	\begin{align*}
	\|u^+\|^2_\mathcal{H} &= \int_{\R\times\T} u_x u^+_x - V(x) u_t u^+_t \\
	& =\underbrace{J'(u)[u^+]}_{=J'(u)[u-u^-]=0} + \int_{\R\times\T} \Gamma(x) |u|^{p-1}uu^+\der (x,t)\\
	& \leq \|\Gamma\|_\infty^\frac{1}{p+1} \left(\int_{\R\times\T} \Gamma(x) |u|^{p+1}\der (x,t)\right)^\frac{p}{p+1} \|u^+\|_{L^{p+1}}\\
	& \leq C\|\Gamma\|_\infty^\frac{1}{p+1} \left(\int_{\R\times\T} \Gamma(x) |u|^{p+1}\der (x,t)\right)^\frac{p}{p+1} \|u^+\|_\mathcal{H}.
	\end{align*}
 	Together with a similar estimate for $u^-$ and $J(u)= \frac{p-1}{p+1} \int_{\R\times\T} \Gamma(x)|u|^{p+1} \der (x,t)$ this implies the second inequality of \eqref{eq:loc:manifold_est}. Since $J(u)\leq C\|u\|_\mathcal{H}^{p+1}$ for $u\in\mathcal{M}$ by \cref{prop:Lp-embedding} and $u\not=0$ we also get the first inequality of \eqref{eq:loc:manifold_est}.

	\medskip
	
	It remains to show the boundedness of a Palais-Smale sequence. Let $(u_n)_{n\in\N}$ be a Palais-Smale sequence for $J$ in $\mathcal{H}$. From the identities 
	\begin{align*}
	-J'(u_n)[u_n^-] &= \|u_n^-\|_\mathcal{H}^2 + \int_{\R\times\T} \Gamma(x) |u_n|^{p-1} u_n u_n^-\der (x,t),  \\
	J'(u_n)[u_n^+] &= \|u_n^+\|_\mathcal{H}^2 - \int_{\R\times\T} \Gamma(x) |u_n|^{p-1} u_n u_n^+\der (x,t)
	\end{align*}
	we find 
	$$
	\|u_n^\pm\|_\mathcal{H}^2 \leq \left( C\|\Gamma\|_\infty^\frac{1}{p+1} \left(\int_{\R\times\T} \Gamma(x)|u_n|^{p+1}\der (x,t)\right)^\frac{p}{p+1} + o(1)\right)\|u_n^\pm\|_\mathcal{H}
$$
so that 
\begin{equation} \label{eq:normbound}
	\|u_n\|_\mathcal{H} \leq C\left( \left(\int_{\R\times\T} \Gamma(x)|u_n|^{p+1}\der (x,t)\right)^\frac{p}{p+1} + 1\right).
\end{equation}
Using the identity $J(u_n)-\tfrac12 J'(u_n)[u_n] = \frac{p-1}{p+1} \int_{\R\times\T} \Gamma(x)|u_n|^{p+1} \der (x,t)$ and \eqref{eq:normbound} we obtain 
\begin{align*}
\int_{\R\times\T} \Gamma(x)|u_n|^{p+1}\der (x,t) &\leq C(J(u_n)+\|u_n\|_\mathcal{H}) \\
& \leq C\left(J(u_n) + \left(\int_{\R\times\T} \Gamma(x)|u_n|^{p+1}\der (x,t)\right)^\frac{p}{p+1} + 1\right).
\end{align*}
This implies
$$
\int_{\R\times\T} \Gamma(x)|u_n|^{p+1}\der (x,t) \leq C(J(u_n)+1)
$$
and once again with the help of \eqref{eq:normbound} 
$$
\|u_n\|_\mathcal{H}^\frac{p+1}{p} \leq C(J(u_n)+1)
$$
which proves boundedness of Palais-Smale sequences as claimed. 
\end{proof}

The next lemma is a variant of P.L.Lions' concentration compactness result, cf.~\cite{Willem}. Since the norm in $\mathcal{H}$ is nonlocal in space, the result cannot be derived in the standard way by combining H\"older- and Sobolev-inequality. Instead, one uses the embedding from $\mathcal{H}$ into another space $H$ defined by
$$
H \coloneqq \{u\in L^2(\R\times\T): u(x,t) = \sum_{k\in \Zodd} \hat u_k(x) e_k(t), \;\hat u_k=\overline{\hat u_{-k}} \;\forall k\in\Zodd, \;\|u\|_H<\infty\} 
$$
equipped with the norm $\|\cdot\|_H$ according to 
$$
\|u\|_H^2 \coloneqq \sum_{k\in \Zodd} \frac{1}{|k|}\|\hat u_k'\|_{L^2}^2+|k|\|\hat u_k\|_{L^2_V}^2. 
$$
The norm in $H$ is local in space and is suited for standard concentration-compactness arguments. 
In \cref{prop:H_embedding} we show boundedness of the embeddings $\calH \embeds H \embeds L^p(\R \times \T)$ for $p \in [2, 4)$.
Next we state the concentration-compactness lemma. A proof using the intermediate space $H$ and its embedding properties, interpolation arguments, and $\R=\bigcup_{x\in2r\Z}[x-r,x+r]$ for $r>0$ can be found in \cite{maier_reichel_schneider}.

\begin{lemma} \label{lem:ConccompLemma}
Let $q\in [2,\infty)$ and $r>0$ be given and let $(u_n)_{n\in\N}$ be a bounded sequence in $\mathcal{H}$ such that
\begin{align*}
\sup_{x\in 2r\Z} \int_{[x-r,x+r]\times\T} |u_n|^q \der (x,t) \to 0 \text{ as } n\to\infty.
\end{align*}
Then $u_n\to 0$ in $L^{\tilde{q}}(\R\times\T)$ as $n\to\infty$ for all $\tilde{q}\in (2,\infty)$.
\end{lemma}

\begin{proposition} \label{prop:H_embedding} Suppose that \ref{ass:bounded}, \ref{ass:periodic_infinity}, \ref{ass:spectrum} are fulfilled. Then the following holds true:
\begin{enumerate}
    \item[(i)] The embedding $\iota^\star: \mathcal{H}\to H$ is continuous.
    \item[(ii)] For every $q\in [2,4)$ the embedding $\tilde\iota: H \to L^q(\R\times \T)$ is continuous and locally compact, i.e., $\tilde\iota: H \to L^q(A\times \T)$ is compact for every compact set $A\subset \R$.
\end{enumerate}
\end{proposition}
\begin{proof} (i): By assumption \ref{ass:spectrum} we have $\delta\coloneqq\inf\{|\sqrt{\lambda} - k\omega|: \lambda\in \sigma(L), k\in \Nodd\}>0$ and hence
$$
|\lambda-k^2\omega^2| = |\sqrt{\lambda}-|k|\omega||\sqrt{\lambda}+|k|\omega|\geq \delta |k|\omega
$$
for all $\lambda\in\sigma(L)$ and all $k\in \Zodd$. Thus $|L_k|\geq \delta\omega |k|\id$ on $H^1(\R)$ and for $v\in H^1(\R)$ we therefore have 
\begin{equation} \label{eq:teil1}
\|v\|_{\mathcal{H}_k}^2 = \langle \sqrt{|L_k|}v,\sqrt{|L_k|}v\rangle_{L^2_V}  \geq \delta \omega|k|\|v\|_{L^2_V}^2. 
\end{equation}
For $C>0$ small enough and for $v\in \mathcal{H}_k^+$, $k\in \Zodd$, we therefore find
$$
\int_\R |v'|^2-k^2\omega^2 \abs{v}^2 V(x) \der x = \|v\|_{\mathcal{H}_k}^2 \geq \frac{\omega^2 k^2C}{|k|-C} \|v\|_{L^2_V}^2
$$
which implies 
$$
\left(1 - \frac{C}{|k|}\right)\int_R |v'|^2\der x \geq k^2\omega^2 \|v\|_{L^2_V}^2
$$
and hence by rearranging terms 
\begin{equation} \label{eq:teil2}
\|v\|_{\mathcal{H}_k}^2\geq \frac{C}{|k|}\|v'\|_{L^2}^2.
\end{equation}
For $v\in \mathcal{H}_k^-$ we have $\int_\R |v'|^2-k^2\omega^2 \abs{v}^2 V(x)\der x\leq 0$ so that $\|v'\|_{L^2}^2\leq k^2\omega^2 \|v\|_{L^2_V}^2$ and hence in this case \eqref{eq:teil2} follows from \eqref{eq:teil1}. 
Having established \eqref{eq:teil2} separately for $v^+ \in \calH_k^+$ and $v^- \in \calH_k^-$, we obtain the estimate \eqref{eq:teil2} for $v \in \calH_k$:
\begin{align*}
	\frac{C}{2 |k|}\|v'\|_{L^2}^2
	\leq \frac{C}{|k|} \bigl( \|(v^+)'\|_{L^2}^2 + \|(v^-)'\|_{L^2}^2 \bigr)
	\leq 
	\|v^+\|_{\mathcal{H}_k}^2 + 
	\|v^-\|_{\mathcal{H}_k}^2
	=
	\|v\|_{\mathcal{H}_k}^2.
\end{align*}
Then \eqref{eq:teil1} and \eqref{eq:teil2} for all $k\in\Zodd$ imply the continuity of the embedding $\iota^\star: \mathcal{H}\to H$. 

\medskip

    (ii): First we see that for $u(x,t) = \sum_{k\in \Zodd} \hat u_k(x) e_k(t)$ we get 
    $$
    \|u\|_{L^q(\R\times\T)} \leq C_q \left(\sum_{k\in\Zodd} \|\hat u_k\|_{L^q}^{q'} \right)^{\nicefrac{1}{q'}}
    $$
    for all $q\in [2,\infty]$ by standard Riesz-Thorin interpolation. With the Gagliardo-Nirenberg inequality for $2<q<4$ and $\theta=\frac{1}{2}-\frac{1}{q}$
\begin{align*}
\|v\|_{L^q} \leq C_{GN} \|v'\|_{L^2}^\theta \|v\|_{L^2}^{1-\theta}, \qquad v\in H^1(\R)
\end{align*}
and a triple H\"older inequality with exponents $\frac{4(q-1)}{q-2}$, $\frac{4(q-1)}{q+2}$, and $\frac{2(q-1)}{q-2}$ we obtain 
\begin{align*}
\|u\|_{L^q(\R\times\T)}^{q'} & \leq  
C \sum_{k\in\Zodd} \|u_k\|_{L^q}^{q'} \leq C \sum_{k\in\Zodd} \|u_k'\|_{L^2}^{q'(\frac{1}{2}-\frac{1}{q})} \|u_k\|_{L^2}^{q'(\frac{1}{2}+\frac{1}{q})} \\
& =  C \sum_{k\in\Zodd} (|k|^{-\frac{1}{2}}\|u_k'\|_{L^2})^{q'(\frac{1}{2}-\frac{1}{q})} (|k|^\frac{1}{2} \|u_k\|_{L^2})^{q'(\frac{1}{2}+\frac{1}{q})} |k|^\frac{-1}{q-1}\\
&  \leq C\left(\sum_{k\in\Zodd} |k|^{-1} \|u_k'\|_{L^2}^2\right)^\frac{q-2}{4(q-1)} \left(\sum_{k\in\Zodd} |k|\|u_k\|_{L^2}^2\right)^\frac{q+2}{4(q-1)} \left(\sum_{k\in\Zodd} |k|^\frac{-2}{q-2}\right)^\frac{q-2}{2(q-1)} \\ 
&\leq  C\|u\|_H^{q'}. 
\end{align*}
Note that $\sum_{k\in\Zodd} |k|^\frac{-2}{q-2}$ converges due to $2<q<4$. This established the continuity of the embedding $\tilde\iota: H \to L^q(\R\times \T)$. For $q=2$ the embedding is clear. The local compactness can be seen as follows. First, we modify $\tilde\iota$ by setting $\tilde\iota_K\coloneqq \tilde\iota\circ T_K$ for $K\in\N$, where $T_K$ truncates the Fourier series to modes $k\in \Zodd$ with $|k|\leq K$. Then $\tilde\iota_K\to \tilde\iota$ in the operator norm as $K\to \infty$. Since $\tilde\iota_N$ maps compactly into $L^q(A\times\T)$ for every compact set $A\subset \R$, the same holds for $\tilde\iota$. 
\end{proof}

We are now ready to give the proof of Theorem~\ref{thm:main}. We employ the abstract result Theorem~35 from \cite{szulkin_weth}, which provides a Palais-Smale sequence $(u_n)_{n\in\N}$ for the functional $J$ that belongs to $\mathcal{M}$ and is minimizing for $J|_\mathcal{M}$. This is possible since by Lemma~\ref{lem:property_J_M} the conditions (B1), (B2) and (i) and (ii) of Theorem~35 in \cite{szulkin_weth} are fulfilled. At this point, we do not claim that also (iii) of Theorem~35 in \cite{szulkin_weth}, which leads to the Palais-Smale condition, holds. But as we shall see this is indeed the case in case (a) of Theorem~\ref{thm:main} but not necessarily in case (b). As a consequence of Theorem~35 in \cite{szulkin_weth} we obtain a minimizing Palais-Smale $(u_n)_{n\in \N}$ in $\mathcal{M}$ with $J'(u_n)\to 0$ as $n\to\infty$. By Lemma~\ref{lem:boundedness_of_PS} we know that $(u_n)_{n\in\N}$ is bounded, and hence there is $u\in\mathcal{H}$ and a subsequence (again denoted by $(u_n)_{n\in\N}$) such that $u_n\rightharpoonup u$ as $n\to\infty$. Using that $u_n\to u$ in $L^{p+1}_\loc(\R\times\T)$ and that compactly supported functions are dense in $\mathcal{H}$ (cf. \cref{rem:dense}) we deduce that $J'(u)=0$.

It remains to shows that the limit function $u$ belongs to $\mathcal{M}$ (which implies that it is non-zero) and that it is a minimizer of $J$ on $\mathcal{M}$. For this purpose we give different proofs depending on which of the assumptions (a) or (b) of \cref{thm:main} holds.

\medskip

Case (a): By Lemma~\ref{lem:property_J_M}(iv) also property (iii) of Theorem~35 in \cite{szulkin_weth} holds and hence it provides a (necessarily nontrivial) ground state of $J|_\mathcal{M}$ and infinitely many bound states. 

\medskip

Case (b) -- purely periodic case: Let us first assume $\Gamma_\loc\equiv 0$ so that we are in a purely periodic setting. In this case property (iii) of Theorem~35 in \cite{szulkin_weth} does not hold. However, due to the periodic structure, concentration-compactness arguments can be used, and the proof is essentially the same as in \cite{maier_reichel_schneider}. Therefore, we only sketch the argument. First, since $u_n \in \calM$ we have $J(u_n)= \frac{p-1}{p+1}\int_{\R\times\T} \Gamma(x) |u_n|^{p+1} \der(x,t)$, and using \cref{lem:boundedness_of_PS} it follows that $0$ is not a limit point of $(u_n)_{n\in\N}$ in $L^{p+1}(\R\times\T)$. Then, by \cref{lem:ConccompLemma} with $r=\frac X2$ there exist $x_n\in X\Z$ and $\delta>0$ such that for $v_n(x,t)\coloneqq u(x+x_n,t)$ we have
$$
\int_{[-\frac{X}{2},\frac{X}{2}] \times \T} |v_n|^2\der (x,t)\geq\delta>0.
$$
The sequence $(v_n)_{n\in\N}$ is still a Palais-Smale sequence for $J$, and hence (by \cref{lem:boundedness_of_PS,prop:Lp-embedding}) has a weakly convergent subsequence $v_n\rightharpoonup v\not=0$ as $n\to\infty$ with $v\in\mathcal{H}$.
Following our observations at the beginning of the proof, we find that $J'(v)=0$ and due to $0=J'(v)[v]=\|v^+\|_\mathcal{H}^2-\|v^-\|_\mathcal{H}^2-(p+1)J_1(v)$ we see that $v^+\not =0$. Hence, $v\in\mathcal{M}$ and by a Fatou-type argument we see that $v$ is indeed a minimizer of $J$ on $\mathcal{M}$. 

\medskip

Case (b) -- perturbed periodic case: Now we assume $\Gamma_\loc\not\equiv 0$. In this case let us consider two functionals: next to $J=J_0-J_1$ we also consider $J^\per = J_0-J_1^\per$ with 
$$
J_1^\per(u) = \tfrac{2}{p+1}\int_{\R\times\T} \Gamma_\per (x) |u|^{p+1}\der (x,t)
$$
The only difference between the two functionals is that $J_1^\per \leq J_1$ due to the assumption that $\Gamma_\per \leq \Gamma=\Gamma_\per+\Gamma_\loc$. While the Hilbert space $\mathcal{H}$ and the decomposition $\calH = \calH^+ \oplus \calH^-$ stay the same for both functionals, the underlying Nehari-Pankov manifolds are different, i.e., next to $\mathcal{M}$ we also have 
$$
\mathcal{M}^\per = \{ u\in \mathcal{H}\setminus\mathcal{H}^-: {J^\per}'(u)[w] = 0 \mbox{ for all }
 w\in [u]+\mathcal{H}^-\}
$$
together with the two ground-state levels
$$
c^\per = \inf\{ J^\per(u): u \in \mathcal{M}^\per\}, \quad  c = \inf\{J(u): u \in \mathcal{M}\}.
$$
We already know from the previous case that $c^\per$ is attained for a minimizer $u^\per$ with ${J^\per}'(u^\per) = 0$. Let us show that $c \leq c^\per$. Associated to $\calM$ and $\calM^\per$ we have the two maps $m\colon \calH\setminus\calH^-\to \calM$ and $m^\per \colon  \calH\setminus\calH^-\to \calM^\per$, cf. \cref{lem:property_J_M}.
Since $m^\per(u^\per) = u^\per \not \in \mathcal{H}^-$ and $u \coloneqq m(u^\per) \in [0, \infty) u^\per+\mathcal{H}^-$, \cref{lem:property_J_M}~(v) shows $J^\per(u) \leq J^\per(u^\per)=c^\per$. On the other hand, $\Gamma_\loc\geq 0$ implies $J(u) \leq J^\per(u)$. Combining these, we have $c \leq J(u) \leq J^\per(u) \leq J^\per(u^\per) = c^\per$ as claimed.

Let us first consider the case where $c = c^\per$. For the above inequalities to be equalities, $\Gamma_\loc u = 0$ and $u = u^\per$ must hold. But then $u = u^\per$ is both a minimizer and a solution since $J'(u) = {J^\per}'(u^\per) = 0$. So let us assume in the following that $c < c^\per$ holds. Based on this inequality let us now verify that $c$ is attained. As in the purely periodic case we start with the bounded Palais-Smale sequence $(u_n)_{n\in\N}$ for $J$ on $\mathcal{M}$ where $0$ is not a limit point of $(u_n)_{n\in\N}$ in $L^{p+1}(\R\times\T)$. By Lemma~\ref{lem:ConccompLemma} this time for the exponent $q=p+1$ there exists a sequence $x_n\in X\Z$ and $\delta>0$ such that
$$
\int_{[x_n-\frac{X}{2},x_n+\frac{X}{2}] \times \T} |u_n|^{p+1} \der (x,t)\geq \delta
$$
for all $n\in \N$. We claim that due to $c<c^\per$ we have that $0$ is not a limit point of $(u_n)_{n\in\N}$ in $L^{p+1}_\loc(\R\times\T)$, which is enough to conclude that a weakly convergent subsequence $u_n\rightharpoonup u$ provides a minimizer $u$ of $J$ on $\mathcal{M}$ (the multiplicity result is then the same as in the purely periodic case). So let us a assume for contradiction that a subsequence of $(u_n)_{n\in\N}$ (again denoted by $(u_n)_{\in\N}$) converges to $0$ in $L^{p+1}_\loc(\R\times\T)$. Then necessarily $|x_n|\to \infty$ as $n\to\infty$. If we set $v_n(x,t) \coloneqq u_n(x-x_n,t)$ then (up to a subsequence) $v_n\rightharpoonup v$ in $\mathcal{H}$ and $v_n \to v$ in $L^{p+1}_\loc(\R\times\T)$ and pointwise a.e. as $n\to\infty$ for some $v\in\mathcal{H}\setminus\{0\}$. Recalling Remark~\ref{rem:dense} let us take a function $\varphi\in \mathcal{M}$ with compact support and set $\varphi_n(x,t) \coloneqq \varphi(x+x_n,t)$. Then 
\begin{align*}
o(1) = J'(u_n)[\varphi_n] & = J_0'(u_n)[\varphi_n] - {J_1^\per}'(u_n)[\varphi_n] - 2\int_{\R\times\T} \Gamma_\loc(x) |u_n|^{p-1} u_n\varphi_n \der (x,t) \\
& = J_0'(v_n)[\varphi] - {J_1^\per}'(v_n)[\varphi] -2\int_{\supp\varphi} \Gamma_\loc(x-x_n) |v_n|^{p-1} v_n\varphi\der (x,t) \\
& \to {J^\per}'(v)[\varphi]
\end{align*}
where we have used that  $v_n^\pm\rightharpoonup v^\pm$, $v_n\to v$ in $L^{p+1}_\loc(\R\times\T)$, $\varphi$ has compact support and $\Gamma_\loc(x)\to 0$ as $|x|\to \infty$. Hence $v$ is a nontrivial critical point of $J^\per$ and belongs to $\mathcal{M}^\per$. Thus 
\begin{align*}
    c^\per\leq J^\per(v)& = J^\per(v)-\frac{1}{2} {J^\per}'(v)[v] \\
    &= \frac{p+1}{p-1} \int_{\R\times\T} \Gamma^\per(x) |v|^{p+1} \der (x,t) \\
    & \leq \frac{p+1}{p-1} \liminf_{n\in\N} \int_{\R\times\T} \Gamma^\per(x) |v_n|^{p+1} \der (x,t) \\
    & \leq \frac{p+1}{p-1} \liminf_{n\in\N} \int_{\R\times\T} \Gamma(x) |v_n|^{p+1} \der (x,t) \\
    & = \liminf_{n\in\N} \left(J(u_n)-\frac{1}{2} J'(u_n)[u_n]\right) \\
    & = c
\end{align*}
which contradicts the already established inequality $c<c^\per$. This contradiction completes the proof of the perturbed periodic case.

\medskip

The multiplicity result can finally be obtained as follows: instead of restricting to $\nicefrac{T}{2}$-antiperiodic functions, we can consider the set of $\nicefrac{T}{2m}$-antiperiodic functions for $m\in \Nodd$. Note that $u$ being $\nicefrac{T}{2m}$-antiperiodic implies that also $|u|^{p-1}u$ is $\nicefrac{T}{2m}$-antiperiodic. Hence, $\nicefrac{T}{2m}$-antiperiodicity is compatible with the underlying equation \eqref{eq:basic}. Assuming $\nicefrac{T}{2m}$-antiperiodicity amounts to allowing temporal Fourier-modes in the set $m\Zodd$ instead of $\Zodd$. We can now find a ground state $u_m$ of $J$ in the set of $\nicefrac{T}{2m}$-antiperiodic functions for every $m\in \Nodd$. The set $\{u_m: m\in \Nodd\}$ cannot be finite since the (in absolute value) lowest index of all non-zero Fourier modes of $u_m$ is at least $m$.

To finish the proof of Theorem~\ref{thm:main} we show that our obtained critical  point $u \in \mathcal{H}$ of $J$ lies in $H^2(\R \times \T)$ and satisfies the equation \eqref{eq:basic} pointwise almost everywhere.
The proof is based on the ideas of Chapter 6 in \cite{maier_reichel_schneider}. We present in detail only the parts which differ from \cite{maier_reichel_schneider}. \medskip

We start with analyzing the linear problem
\begin{align*}
    V(x) \partial_t^2 w - \partial_x^2 w = f .
\end{align*}
\begin{lemma}\label{lem:reglin}
Let $f \in \mathcal{H}'$. Then there exists a unique solution $w \in \mathcal{H}$ to
\begin{align}\label{eq:linprob}
    \sum_{k\in \Zodd} b_k(w_k,\varphi_k) 
	= \langle f, \varphi \rangle_{\mathcal{H}' \times \mathcal{H}} \quad \text{ for all } \varphi \in \mathcal{H}
\end{align}
and  $\norm{w}_\mathcal{H}=\norm{f}_{\mathcal{H}'}$ holds.
Further, if $f \in L^2_V(\R \times \T)$ then $w\in H^1(\R\times\T)$. 
\end{lemma}
\begin{proof}
The first statement follows directly from Lemma 6.1 in \cite{maier_reichel_schneider}. 
For the proof of the second statement we use the spectral resolution $L=\int_\R \lambda \der P_\lambda$ to expand 
\begin{align*}
	f(x, t) =\sum_{k \in \Zodd} \ee^{\ii \omega k t} \left(\int_{\R} \der P_\lambda[\hat f_k]\right)(x)
\end{align*}
and the solution $w$ to \eqref{eq:linprob} as
\begin{align*}
	w(x,t) =\sum_{k \in \Zodd} \ee^{\ii \omega k t} \left(\int_{\R} \frac{1}{\lambda - k^2 \omega^2} \der P_\lambda[\hat f_k]\right) (x)
\end{align*}
so that
\begin{align*}
	\partial_t w(x, t) = \sum_{k \in \Zodd}  \ee^{\ii \omega k t} \left( \int_{\R} \frac{\ii \omega k}{\lambda - k^2 \omega^2} \der P_\lambda[\hat f_k] \right)(x)
\end{align*}
By our spectral assumption \ref{ass:spectrum} there exists a constant $c >0$ such that
\begin{align*}
    \frac{|\omega k|}{|\lambda - k^2 \omega^2|}= \frac{\left| \omega k\right|}{|\sqrt{\lambda}- |k| \omega| \left|\sqrt{\lambda}+ |k| \omega\right|} \leq c \frac{|\omega k|}{\left|\sqrt{\lambda}+ |k| \omega\right|} \leq c
\end{align*}
for all $k \in \Zodd$ and $\lambda \in \sigma(L)$. We conclude
\begin{align*}
    \norm{ \partial_t w}_{L^2_V(\R \times \T)}^2 &= \sum_{k \in \Zodd} \int_\R \frac{\omega^2 k^2}{(\lambda-k^2 \omega^2)^2 } \der \norm{P_\lambda \hat f_k}_{L^2_V(\R)}^2  \\
    & \leq \sum_{k \in \Zodd} \int_\R c^2 \der \norm{P_\lambda \hat{f}_k}_{L^2_V(\R)}^2 =  c^2 \norm{f}_{L^2_V(\R \times \T)}^2.
\end{align*}
Next, 
\begin{align*}
    \sum_{k\in \Zodd}\int_\R |\hat w_k'|^2 -k^2\omega^2 V(x)|\hat w_k|^2 \der x  = \norm{w^+}_{\mathcal{H}}^2-\norm{w^-}_{\mathcal{H}}^2 \leq \norm{w^+}_{\mathcal{H}}^2+ \norm{w^-}_{\mathcal{H}}^2=\norm{w}_\mathcal{H}^2
\end{align*}
implies
\begin{align*}
    \sum_{k\in \Zodd}\int_\R |\hat w_k'|^2 \der x &\leq \sum_{k\in \Zodd}\int_\R k^2\omega^2 V(x)|\hat w_k|^2 \der x + \norm{w}_\mathcal{H}^2
	=\norm{\partial_t w}_{L^2_V(\R \times \T)}^2 + \norm{w}_\mathcal{H}^2
\end{align*}
and hence
\begin{align*}
    \norm{ \partial_x w}_{L^2(\R \times \T)}^2 = \sum_{k \in \Zodd} \norm{\hat{w}_k'}_{L^2(\R)}^2 < \infty. &\qedhere
\end{align*}
\end{proof}

Weak solutions to \eqref{eq:basic} are defined as follows.

\begin{definition} \label{def:weak_sol} A time periodic function $u\in H^1(\R\times\T)$ is called a weak solution of \eqref{eq:basic} if
\begin{equation} \label{eq:weak}
\int_{\R\times\T} \left(u_x\phi_x-V(x) u_t\phi_t - \Gamma(x)\abs{u}^{p-1}u\phi\right) \der (x,t)=0
\end{equation}
holds for every time periodic $\phi\in H^1(\R\times\T)$. 
\end{definition}

Next we transfer the result from Lemma \ref{lem:reglin} to the nonlinear case and show that critical points of $J$ are weak solutions to \eqref{eq:basic}.
\begin{lemma}
    Let $u \in \mathcal{H}$ be a critical point of $J$. Then $u\in H^1(\R\times\T)$ is a weak solution to \eqref{eq:basic} in the sense of Definition \ref{def:weak_sol}.
\end{lemma}
\begin{proof}
Since $u\in \mathcal{H}$ is a critical point of $J$ it satisfies \eqref{eq:linprob} with $f=\Gamma(x)|u|^{p-1}u$. Moreover, since $u \in L^q_V(\R \times \T)$ for all $q \in [2,\infty)$ by Proposition \ref{prop:Lp-embedding} and $\Gamma \in L^\infty(\R)$ by assumption \ref{ass:bounded} we see that $f \in L^2_V(\R \times \T)$. Then Lemma~\ref{lem:reglin} applies and we obtain $u\in H^1(\R\times\T)$. Clearly, \eqref{eq:weak} holds for all $\phi\in \mathcal{H}$, which consists of $\frac{T}{2}$-antiperiodic functions. By density, \eqref{eq:weak} holds for all $\frac{T}{2}$-antiperiodic functions in $H^1(\R\times\T)$. Finally, we note that \eqref{eq:weak} trivially holds for all $\frac{T}{2}$-periodic functions in $H^1(\R\times\T)$ since then all three integrands are products of a $\frac{T}{2}$-antiperiodic function with a $\frac{T}{2}$-periodic function and thus integrate to $0$. This finishes the proof of the lemma.
\end{proof}
\begin{lemma}
    Let $u \in \mathcal{H}$ be a weak solution to \eqref{eq:basic}. Then  $u \in H^2(\R \times \T)$ satisfies \eqref{eq:basic} pointwise almost everywhere. 
\end{lemma}
\begin{proof}
By using difference quotients as is Lemma~6.3 and Lemma~6.4 in \cite{maier_reichel_schneider} we infer $ \partial_t^2 u, \partial_x \partial_t u \in L^2(\R \times \T)$. 
As in the proof of Lemma 6.5 in \cite{maier_reichel_schneider} we conclude $\partial_x^2u \in L^2(\R \times \T)$.
\end{proof}

%% file: functional_calc.tex

\section{\texorpdfstring{$L^p$}{Lp}-embeddings}
\label{sec:embeddings}

This section is devoted to the proof of the following embedding property for the Hilbert space $\calH$.

\begin{proposition} \label{prop:Lp-embedding} Let \ref{ass:bounded}, \ref{ass:periodic_infinity}, \ref{ass:spectrum}, and \ref{ass:point_spectrum} be fulfilled. Then for every $p\in [2,\infty)$ the embedding $\calH \embeds L^p(\R\times \T)$ is bounded and locally compact, i.e., $\mathcal{H} \embeds L^p(A\times \T)$ is compact for every compact set $A \subset \R$. 
\end{proposition}

As $\calH$ depends strongly on the spectral projections of the operator $L$, we first develop a functional calculus for the operator $L$. 
In Section \ref{subsec:functional_calc} we give a general description of the functional calculus, cf. \cref{thm:functional_calculus:general}. 
In Section \ref{subsec:spectral_measure} we develop a description of the associated spectral measure. This ends in \cref{thm:spectral_measure:explicit}, where we calculate the density of the spectral measure  with respect to the sum of the Lebesgue measure (for the essential spectrum) and the counting measure (for the point spectrum). 
Lastly, in Section \ref{subsec:embeddings} we consider $L^p$-embeddings in a slightly more general setting. Using uniform bounds on (generalized) eigenfunctions and estimates on the spectrum, we show the embedding result \cref{thm:Lp-embedding:general}, of which \cref{prop:Lp-embedding} is a special case (with $\alpha = \beta = 0$). Throughout this section, we will always assume that the potential $V$ satisfies \ref{ass:bounded} and \ref{ass:periodic_infinity}.

\subsection{A functional calculus for \texorpdfstring{$L$}{L}}\label{subsec:functional_calc}

We describe a functional calculus for the spectral problem 
\begin{align}\label{eq:evp}
	-u'' = \lambda V(x) u
	\qquad\text{for}\quad
	x \in \R.
\end{align}
Recall that $V \in L^\infty(\R; \R)$ satisfies $\essinf_\R V > 0$.
We follow \cite{coddington} and begin with some preliminary notation.
\begin{notation}
	We denote the upper half-plane by $\H \coloneqq \set{z \in \C \colon \Im[z] > 0}$, the unit circle by $\S \coloneqq \set{z \in \C \colon \abs{z} = 1}$, the unit circle except two points by $\SwR = \S \setminus \set{-1, 1}$. By a \emph{solution} to \eqref{eq:evp} we mean a function $u \in W^{2,1}_\loc(\R; \C)$ solving \eqref{eq:evp} pointwise almost everywhere. 
	
\end{notation}
\begin{definition}
	For $\lambda \in \C$, we denote by $\Psi_1(x; \lambda), \Psi_2(x; \lambda)$ the solutions to \eqref{eq:evp} with initial data
	\begin{align*}
		\Psi_1(0; \lambda) = 1, 
		\quad
		\Psi_1'(0; \lambda) = 0,
		\qquad
		\Psi_2(0; \lambda) = 0, 
		\quad 
		\Psi_2'(0; \lambda) = 1
	\end{align*}
    and write $\Psi(x;\lambda) = \begin{pmatrix} \Psi_1(x;\lambda) \\ \Psi_2(x;\lambda) \end{pmatrix}$.
	We further define the Wronskian 
	\begin{align*}
		W(f,g) \coloneqq f g' - f' g
		\qquad\text{so that}\qquad
		W(f,g) \big\vert_a^b 
		=\int_a^b \left(L f \cdot g - f \cdot L g\right) \der[V] x.
	\end{align*}
	In particular, $W(f, g)$ is constant if $f, g$ both solve $L u = \lambda u$. 
\end{definition}

The following three results about the solutions of \eqref{eq:evp} can be found in \cite[Chapter 9]{coddington} for the case $V=1$ and with $-\dx p(x) \dx + q(x)$ in place of $-\dxsquare$, but proofs can easily be adapted to the weighted setting. 
More precisely, the following three results correspond to Theorem~2.4, Theorem~2.3, and Section~5 in \cite{coddington}. 
Related is also the work \cite{bennewitz_everitt_chapter} where the authors treat the weighted setting on the half-line.

\begin{theorem} \label{thm:limit_point}
	$L$ is of ``limit-point type'' at $+\infty$ and $-\infty$. That is, for any $\lambda \in \C$ at most one linearly independent solution of \eqref{eq:evp} lies in $L_V^2([0, \infty))$, and the same holds for $L_V^2((-\infty, 0])$. 
\end{theorem}
\begin{theorem}\label{thm:existence_of_m}
	There exist holomorphic functions $m_\pm \colon \C\setminus\R \to \C$ such that for each $\lambda \in \C \setminus \R$, the solution
	\begin{align*}
		u = \Psi_1(\impvar; \lambda) + m_\pm(\lambda) \Psi_2(\impvar; \lambda) 
	\end{align*}
	lies in $L^2(\pm [0, \infty))$. Moreover, we have the identity
	\begin{align*}
		\Im[m_\pm(\lambda)] = \Im[\lambda] \int_0^{\pm\infty} \abs{u}^2 \der[V]{x}.
	\end{align*}
\end{theorem}

Next we present the definition of a functional calculus for the operator $L$. It uses the Hilbert space $L^2(\mu)$ consisting of functions $\R \to \C^2$ which are square-integrable with respect to a matrix-valued function $\mu$ defined on $\R$. See \cref{sec:vector_L2} for a definition and basic properties of $L^2(\mu)$.

\begin{theorem}\label{thm:functional_calculus:general}
	There exists an increasing function $\mu \colon \R \to \R^{2 \times 2}$, called the \emph{spectral measure}, such that the map $T \colon L^2_V(\R; \C) \to L^2(\mu)$ given by
	\begin{align*}
		T[f](\lambda) 
		\coloneqq \int_\R f(x) \Psi(x; \lambda) \der[V]x
	\end{align*}
	for compactly supported $f \in L^2_V(\R; \C)$ is an isometric isomorphism, with inverse given by
	\begin{align*}
		T^{-1}[g](x)
		= \int_\R g_i(\lambda) \Psi_j(x; \lambda) \der \mu_{ij}(\lambda)
	\end{align*}
	for compactly supported $g \in L^2(\mu)$, where we use Einstein summation convention.
	Moreover, at points $\lambda_1, \lambda_2 \in \R$ where $\mu$ is continuous the increment of the spectral measure can be computed as
	\begin{align*}
		\mu(\lambda_2) - \mu(\lambda_1)
		= \lim_{\eps \to 0+} \int_{\lambda_1}^{\lambda_2} M(s + \ii \eps) \der s
	\end{align*}
	where
	\begin{align*}
		M \coloneqq \frac{1}{\pi} \Im \Bigl[ (m_- - m_+)^{-1} \begin{pmatrix}
			1 & \frac12 (m_- + m_+) 
			\\ \frac12 (m_- + m_+) & m_- m_+
		\end{pmatrix} \Bigr].
	\end{align*}
\end{theorem}

Let us show that $T$ as above diagonalizes the differential operator $L$.

\begin{lemma}\label{lem:calculus_diagonalizes_L0}
	Let $f \in L^2_V(\R; \C)$. Then $f \in H^2(\R; \C)$ if and only if $\Id_\R \cdot T[f] \in L^2(\mu)$. In this case, $T[L f] = \Id_\R \cdot T[f]$ holds. 
\end{lemma}
\begin{proof}
	\textit{Part 1:}
	First let $f \in H^2(\R; \C)$ be compactly supported. Then 
	\begin{align*}
		T[L f](\lambda)
		&= \int_\R L f(x) \cdot \Psi(x; \lambda) \der[V]x
		= \int_\R f(x) \cdot L \Psi(x; \lambda) \der[V]x
		\\ &= \lambda \int_\R f(x) \Psi(x; \lambda) \der[V]{x}
		= \lambda T[f](\lambda).
	\end{align*}
	Thus $\Id_\R \cdot T[f] \in L^2(\mu)$ and $T[L f] = \Id_\R\cdot T[f]$. For general $f$, we argue by approximation.

	\textit{Part 2:} 
	Now assume that $\lambda T[f] \in L^2(\mu)$, and let $g \in H^2(\R;\C)$. Then we have
	\begin{align*}
		\int_\R f \cdot L g \der[V]{x}
		&= \int_\R T_i[f] (\lambda) \cdot \overline{T_j[L g] (\lambda)} \der \mu_{ij}(\lambda)
		= \int_\R T_i[f] (\lambda) \cdot \overline{\lambda T_j[g](\lambda)} \der \mu_{ij}(\lambda)
		\\ &= \int_\R \lambda T_i[f] (\lambda) \cdot \overline{T_j[g] (\lambda)} \der \mu_{ij}(\lambda)
		= \int_\R T^{-1}[\Id_\R \cdot T[f]] \cdot g \der[V]{x}.
	\end{align*}
	Since $L$ is self-adjoint, we have $f \in H^2(\R; \C)$ and $L f = T^{-1}[\Id_\R\cdot T[f]]$.
\end{proof}

\subsection{Description of the spectral measure}\label{subsec:spectral_measure}
Recall that $V$ is periodic on $[R^+, \infty)$ with period $X^+$ and also on $(-\infty, R^-]$ with period $X^-$. In this section, we use this property to give a better description of the spectral measure $\mu$ in \cref{thm:functional_calculus:general}. 

\begin{remark}
	One can also give a description of the spectral measure in the way we do below when $V$ is asymptotically periodic at $\pm\infty$ provided that the defect from the periodic limiting profiles is integrable.
	We avoid this because it creates additional difficulties when considering $L^p$-embeddings.
\end{remark}

\begin{definition}
	Define the propagation matrix
	\begin{align*}
		P(y,x; \lambda) = \begin{pmatrix}
			\Psi_1(y; \lambda) & \Psi_2(y; \lambda)
			\\ \Psi_1'(y; \lambda) & \Psi_2'(y; \lambda)
		\end{pmatrix}
		\cdot 
		\begin{pmatrix}
			\Psi_1(x; \lambda) & \Psi_2(x; \lambda)
			\\ \Psi_1'(x; \lambda) & \Psi_2'(x; \lambda)
		\end{pmatrix}^{-1}
	\end{align*}
	so that any solution $u$ of \eqref{eq:evp} satisfies
	\begin{align*}
		\begin{pmatrix}
			u(y) \\ u'(y)
		\end{pmatrix} = P(y, x; \lambda) \begin{pmatrix}
			u(x) \\ u'(x)
		\end{pmatrix}
	\end{align*}
	for all $x, y \in \R$. Further define the monodromy matrices
	\begin{align*}
		P^\pm(\lambda) = P(0, x; \lambda) P(x\pm X^\pm,0;\lambda), 
	\end{align*}
	where $x \gtrless R^\pm$.
\end{definition}
\begin{remark}
	The monodromy matrix $P^+(\lambda)$ is a propagation matrix along one period expressed in terms of values $\begin{pmatrix}
		u(0) \\ u'(0)
	\end{pmatrix}$ at $x = 0$:
	\begin{align*}
		P^+(\lambda)
		= P(0, x; \lambda) P(x + X^+, x; \lambda) P(0, x; \lambda)^{-1}.
	\end{align*}
	The same is true for $P^-$, except we move one period towards $-\infty$.
\end{remark}

\begin{lemma}\label{lem:monodromy}
	The propagation and monodromy matrices have determinant $1$. For each $\lambda \in \C \setminus \R$, $P^\pm(\lambda)$ has an eigenvalue $\rho$ with $\abs{\rho} < 1$, and the eigenspace is  $\lspan\{\begin{pmatrix}1 \\ m_\pm(\lambda)\end{pmatrix}\}$.
\end{lemma}
\begin{proof}
	We only consider ``$+$''. First, $P^+$ is well-defined since for $y, x > R^+$ and omitting $\lambda$ we have
	\begin{align*}
		P(0, y) P(y + X^+, 0)
		&= P(0, x) P(x, y) P(y + X^+, x + X^+) P(x + X^+, 0)
		\\ &= P(0, x) P(x, y) P(y, x) P(x + X^+, 0)
		= P(0, x) P(x + X^+, 0)
	\end{align*} 
	where $P(y + X^+, x + X^+) = P(y, x)$ because the differential operator is $X^+$-periodic on $[R^+, \infty)$.

	Next $P$, and therefore also $P^+$, has determinant $1$ since
	\begin{align*}
		\dv{}{x} \det \begin{pmatrix}
			\Psi_1(x; \lambda) & \Psi_2(x; \lambda)
			\\ \Psi_1'(x; \lambda) & \Psi_2'(x; \lambda)
		\end{pmatrix}
		= \Psi_1(x; \lambda) \Psi_2''(x; \lambda) - \Psi_1''(x; \lambda) \Psi_2(x; \lambda) = 0
	\end{align*}
    and at $x=0$ the determinant is $1$ due to the definition of $\Psi_1, \Psi_2$.

	Lastly let $\lambda \in \C \setminus \R$. We write $v \coloneqq (1, m_+(\lambda))^\top$ and $u \coloneqq \Psi_1(\impvar; \lambda) + m_+(\lambda) \Psi_2(\impvar; \lambda) = \Psi(\impvar; \lambda) \cdot v$, which is square integrable near $+\infty$ by \cref{thm:existence_of_m}. By definition of $P^+(\lambda)$ we further have for $n \in \N$ and $x > R^+$
    \begin{align*}
        \begin{pmatrix}
            u(x+nX^+) \\ u'(x+nX^+) 
        \end{pmatrix} 
        & = P(x+nX^+,0) v = P(x,0)  P(0,x) P(x+nX^+,x) P(x,0) v  \\
        &= P(x,0) P(0,x) P(x+X^+,x)^n P(x,0)v 
        = P(x,0) P^+(\lambda)^n v
	\intertext{and therefore}
		u(x + n X^+) 
		&= \Psi(x; \lambda) \cdot P^+(\lambda)^n v.
	\end{align*}
	Since the function $u$ is $L^2$-localized, we have that $P^+(\lambda)^n v \to 0$ as $n \to \infty$. 
	As $P^+(\lambda)$ has determinant $1$ and is $2\times2$, we conclude that $v$ must be an eigenvector of $P^+(\lambda)$ to an eigenvalue $\rho$ with $\abs{\rho} < 1$.
\end{proof}

\begin{lemma}\label{lem:rho}
	The monodromy matrices $P^\pm(\lambda)$ are holomorphic on $\C$. As a consequence, the ``singular sets''
	\begin{align*}
		S_\pm \coloneqq \set{\lambda \in \C \colon P^\pm(\lambda) \text{ has eigenvalue } -1 \text{ or }1}
	\end{align*}
	are discrete subsets of $\R$. Moreover, there exist open neighborhoods $D_\pm$ of $\overline{\H}\setminus S_\pm$ and continuous functions $\rho_\pm \colon D_\pm \cup S_\pm \to \C$ that are holomorphic on $D_\pm$ with $\abs{\rho_\pm} < 1$ on $\H$ and such that $\rho_\pm(\lambda)$ is an eigenvalue of $P^\pm(\lambda)$ for $\lambda \in D_\pm \cup S_\pm$. If $I$ is a connected component of $\R \setminus S_\pm$, then one of the following alternatives holds:	
	\begin{enumerate}
		\item $\abs{\rho_\pm} < 1$ on $I$.
		\item $\abs{\rho_\pm} = 1$ on $I$ and $\rho_\pm'(\lambda) \overline{\rho_\pm(\lambda)} \in \ii (0, \infty)$. If $I$ is bounded, $\rho_\pm$ is a diffeomorphism from $I$ to one of the sets $\S \cap \H$ or $\S \cap (-\H)$.
	\end{enumerate}
\end{lemma}
\begin{remark}
	Comparing $\rho_\pm$ of \cref{lem:rho} with $\rho$ of \cref{lem:monodromy} (which depends on $\pm$), we have $\rho_\pm = \rho$ on $\H$, whereas $\rho_\pm$ can be either $\rho$ or $\rho^{-1}$ on $-\H$, depending on the behaviour of $\rho_\pm$ on $\R$. 
\end{remark}
\begin{remark}
	As we will see in \cref{prop:spectral_density} below, intervals of type (a) and (b) correspond to spectral gaps and spectral bands of $L$, respectively.
\end{remark}

\begin{proof}[Proof of \cref{lem:rho}]
	We only present the ``$+$'' case. 
	By \cite[§13, Theorem~III]{walterODE} and arguments therein, $\Psi_j(x; \lambda)$ are holomorphic functions of $\lambda$ for fixed $x$, and in particular $P^+$ are holomorphic. From \cref{lem:monodromy} we know that $P^+(\lambda)$ does not have eigenvalues of modulus $1$ for $\lambda \in \C\setminus\R$. As a consequence of the identity theorem for holomorphic functions, $S_+$ must be discrete.
	
	For $\lambda \in \H$, let $\rho_+(\lambda)$ be the unique eigenvalue of $P^+(\lambda)$ with $\abs{\rho_+} < 1$. Then $\rho_+$ is continuous and can be continuously extended to $\rho_+ \colon \overline{\H} \to \C$. As $\rho_+(\lambda)$ is a simple eigenvalue of $P^+(\lambda)$ for all $\lambda \in \overline{\H}\setminus S_+$, the implicit function theorem applied to $\det(P^+(\lambda) - \rho I) = 0$ shows that $\rho_+$ can be extended to a holomorphic function in a neighborhood $D_+$ of $\overline{\H}\setminus S_+$.
	
	Denote by $C$ the connected component of $D_+\setminus \overline{\H}$ touching all of $I$. Then either $\abs{\rho_+} < 1$ on $C$ or $\abs{\rho_+} > 1$ on $C$. In the first case, the maximum principle shows that $\abs{\rho_+} < 1$ on $I$, whereas in the second case we have $\abs{\rho_+} = 1$ by continuity. 

	Let us assume $\abs{\rho_+} = 1$ on $I$, take $\lambda_0 \in I$, and write
	\begin{align*}
		\rho_+(\lambda) = \rho_+(\lambda_0) + (\lambda - \lambda_0)^n f(\lambda)
	\end{align*}
	where $f$ is holomorphic with $f(\lambda_0) \neq 0$. For all $h \in \H$ and $\eps > 0$ we have 
	\begin{align*}
		1 > \abs{\rho_+(\lambda_0 + \eps h)}^2 
		= 1 + 2 \eps^n \Re[h^n f(\lambda_0) \overline{\rho_+(\lambda_0)}] + \landauO(\eps^{n+1}).
	\end{align*}
	This shows $\Re[h^n f(\lambda_0) \rho_+(\lambda_0)] \leq 0$ for all $h \in \H$, which is only possible when $n = 1$ and $f(\lambda_0) \rho_+(\lambda_0) \in \ii [0, \infty)$, and therefore $f(\lambda_0) = \rho_+'(\lambda_0) \neq 0$.

	As $\rho_+'$ vanishes nowhere on $I$, $\rho_+$ is a diffeomorphism from $I$ to its image $\rho_+(I) \subseteq \SwR$. If $I = (\lambda_1, \lambda_2)$ is bounded, we have $\lambda_1, \lambda_2 \in S_+$ and therefore $\rho_+(\lambda_1), \rho_+(\lambda_2) \in \set{-1, 1}$. This is only possible when $\rho_+(I)$ is either $\S \cap \H$ or $\S \cap (-\H)$.
\end{proof}

As the eigenvalues $\rho_\pm$, and in particular the associated eigenfunctions, are central to the functional calculus, we introduce symbols for them and show a useful identity.

\begin{definition}
	For $\lambda \in D_\pm$, we denote by $v_\pm(\lambda)$ the nonzero eigenvector of $P^\pm(\lambda)$ to the eigenvalue $\rho_\pm(\lambda)$, and let $\phi_\pm(x; \lambda) \coloneqq v_\pm(\lambda) \cdot \Psi(x; \lambda)$ be the associated eigenfunction.
\end{definition}

\begin{remark} \label{rem:l2}
    Note that $\phi_\pm$ solves \eqref{eq:evp} and satisfies $\phi_\pm(x \pm X^\pm; \lambda) = \rho_\pm(\lambda) \phi_\pm(x; \lambda)$ for $x \gtrless R^\pm$, For $\lambda\in \H$ we have by $|\rho_\pm|<1$ that $\phi_\pm \in L^2(\pm [0,\infty))$. Moreover, $v_\pm$ can locally be chosen holomorphic. 
    In order to keep our notation short we omit the $\lambda$-dependency for $\phi_\pm$.
\end{remark}

\begin{lemma}\label{lem:wronskian_identity}
	Let $\lambda \in \R$ with $\rho_\pm(\lambda) \in \SwR$ and $\phi_\pm(x) = v_\pm(\lambda) \cdot \Psi(x; \lambda)$. Then we have
	\begin{align*}
		W(\phi_+, \overline{\phi_+}) 
		= \frac{\rho_+(\lambda)}{\rho_+'(\lambda)} \int_{R^+}^{R^++X^+} \abs{\phi_+}^2 \der[V]{x}
		\quad\text{or}\quad
		W(\phi_-, \overline{\phi_-}) 
		= - \frac{\rho_-(\lambda)}{\rho_-'(\lambda)} \int_{R^--X^-}^{R^-} \abs{\phi_-}^2 \der[V]{x}.
	\end{align*}
\end{lemma}
\begin{proof}
	We only consider the ``$+$'' case. Recall that
	\begin{align*}
		L \phi_+ = \lambda \phi_+, 
		\qquad
		\phi_+(x + X^+; \lambda) = \rho_+(\lambda) \phi_+(x;\lambda)
	\end{align*}
	so in particular $L \partial_\lambda \phi_+ = \phi_+ + \lambda \partial_\lambda \phi_+$ holds.
	Since $L \overline{\phi_+} = \lambda \overline{\phi_+}$ and $\abs{\rho_+(\lambda)} = 1$, it follows that
	\begin{align*}
		\int_{R^+}^{R^+ + X^+} \abs{\phi_+}^2 \der[V]{x}
		&= \int_{R^+}^{R^+ + X^+} \left( L \partial_\lambda \phi_+ - \lambda \partial_\lambda \phi_+ \right) \overline{\phi_+} \der[V] x
		\\ &= \int_{R^+}^{R^+ + X^+} \left(L \partial_\lambda \phi_+ \cdot \overline{\phi_+} - \partial_\lambda \phi_+ \cdot L \overline{\phi_+}\right) \der[V] x
		\\ &= \left. W(\partial_\lambda \phi_+, \overline{\phi_+}) \right|_{R^+}^{R^+ + X^+}
		\\ &= W(\partial_\lambda (\rho_+ \phi_+), \overline{\rho_+ \phi_+})(R^+)
		- W(\partial_\lambda \phi_+, \overline{\phi_+})(R^+)
		\\ &= \partial_\lambda \rho_+ \overline{\rho_+} W(\phi_+, \overline{\phi_+})(R^+)
		+ (\abs{\rho_+^2} - 1) W(\partial_\lambda \phi_+, \overline{\phi_+})(R^+)
		\\ &= \partial_\lambda \rho_+ \overline{\rho_+} W(\phi_+, \overline{\phi_+}).
		\qedhere
	\end{align*}
\end{proof}

We introduce two coefficients and a singular set to describe the interaction between $\phi_+$ and $\phi_-$.

\begin{definition}
	Let $\lambda \in D_\pm$ with $\rho_\pm(\lambda) \in \SwR$. Then, as $\phi_\pm$ is not real-valued, the functions $\phi_\pm, \overline{\phi_\pm}$ span the space of solutions to \eqref{eq:evp}. Therefore, there exist constants $r, t$ such that
	\begin{align*}
		\phi_\mp = r \phi_\pm + t \overline{\phi_\pm}.
	\end{align*}
	We call
	the number $r = r(\phi_\mp; \phi_\pm)$ the \emph{reflection coefficient} and $t = t(\phi_\mp; \phi_\pm)$ the \emph{transmission coefficient}. Also we define $S_0$ as the singular set
	\begin{align*}
		S_0 \coloneqq \set{\lambda \in D_+ \cap D_- \colon \lspan\{v_+(\lambda) \} = \lspan\{v_-(\lambda) \}}
	\end{align*}
	on which $\phi_+$, $\phi_-$ are linearly dependent, and denote the ``large singular set'' by $S \coloneqq S_+ \cup S_- \cup S_0$.
\end{definition}

\begin{remark}\label{rem:S0}
	Let $\lambda \in \H$. Then, as $\lambda$ is not an eigenvalue of $L$, the $L^2(\pm [0, \infty))$-functions $\phi_\pm$ (cf. \cref{rem:l2}) must be linearly independent. 
	That is, $\H \cap S_0 = \emptyset$ holds. By the identity theorem, $S_0$ has no accumulation points in $D_+ \cap D_-$.
\end{remark}

Recalling \cref{thm:functional_calculus:general}, we now calculate the density $\lim_{\eps \to 0} M(\lambda + \ii \eps)$ of the spectral measure $\mu$ for $\lambda$ away from the singularities $S$. With this we describe the essential spectrum; the point spectrum is considered separately afterwards. 

\begin{proposition}\label{prop:spectral_density}
	The limit 
	\begin{align*}
		\lim_{\eps \to 0+} M(\lambda + \ii \eps)
	\end{align*}
	converges uniformly for $\lambda$ in compact subsets of $\R \setminus S$. If we denote the limiting matrix by $M(\lambda)$ then 
	\begin{align*}
		v_i \overline{v_j} M_{ij}(\lambda)
		&= \frac{1}{2 \pi} \left(\frac{\bbone_\set{\abs{\rho_+(\lambda)} = 1} \abs{v \cdot v_-(\lambda)}^2 \abs{\rho_+'(\lambda)}}{\abs{t(\phi_-; \phi_+)}^2 \norm{\phi_+}_{L^2_V(R^+, R^+ + X^+)}^2}
		+ \frac{\bbone_\set{\abs{\rho_-(\lambda)} = 1} \abs{v \cdot v_+(\lambda)}^2 \abs{\rho_-'(\lambda)}}{\abs{t(\phi_+; \phi_-)}^2 \norm{\phi_-}_{L^2_V(R^--X^-,R^-)}^2}\right).
	\end{align*}
	holds for $\lambda \in \R \setminus S$ and $v \in \C^2$.
\end{proposition}

\begin{proof}
	\textit{Part 1:}
	First, let $\lambda \in \H$. By \cref{lem:monodromy} we can write $v_\pm(\lambda) = c_\pm(\lambda) (1, m_\pm(\lambda))^\top$ for some $c_\pm(\lambda) \in \C \setminus\set0$, and we set $w_\pm(\lambda) \coloneqq c_\mp(\lambda)^{-1}(m_\pm(\lambda), -1)^\top$. 
	Using the definition of $M(\lambda)$ in \cref{thm:functional_calculus:general} we calculate
	\begin{align*}
		w_{+,i} \overline{w_{+,j}} M_{ij}(\lambda) &= \frac{\Im[m_+]}{\pi \abs{c_-(\lambda)}^2},
		&w_{+,i} \overline{w_{-,j}} M_{ij}(\lambda) &= 0,
		\\ w_{-,i} \overline{w_{+,j}} M_{ij}(\lambda) &= 0,
		&w_{-,i} \overline{w_{-,j}} M_{ij}(\lambda) &= - \frac{\Im[m_-]}{\pi \abs{c_+(\lambda)}^2}
	\end{align*} 
	Note that the identity
	\begin{align*}
		v = \frac{1}{m_+ - m_-}\left( (v \cdot v_-) w_+ - (v \cdot v_+) w_- \right)
	\end{align*}
	holds for any $v \in \C^2$. Therefore, for $\lambda \in \H$ we obtain
	\begin{align}\label{eq:loc:matrix_formula:1}
		v_i \overline{v_j} M_{ij}(\lambda)
		&= \frac{1}{\pi \abs{m_+ - m_-}^2} \left( \abs{v \cdot v_-(\lambda)}^2 \frac{\Im[m_+]}{\abs{c_-(\lambda)}^2} - \abs{v \cdot v_+(\lambda)}^2 \frac{\Im[m_-]}{\abs{c_+(\lambda)}^2} \right)
	\end{align}
	Before taking the limit, we express \eqref{eq:loc:matrix_formula:1} in simpler terms. For this, we calculate
	\begin{align*}
		\int_0^\infty \abs{\phi_+}^2 \der[V]{x}
		&= \frac{1}{\lambda - \overline{\lambda}} \int_0^\infty \left(L \phi_+ \cdot \overline{\phi_+} - \phi_+ \cdot L \overline{\phi_+} \right) \der[V]{x}
		= \frac{1}{\lambda - \overline{\lambda}} \left. W(\phi_+, \overline{\phi_+}) \right|_0^\infty
		\\ &= \frac{1}{\lambda - \overline{\lambda}} \left( \phi_+'(0) \overline{\phi_+(0)} - \phi_+(0) \overline{\phi_+'(0)} \right) = \abs{c_+(\lambda)}^2 \frac{\Im[m_+]}{\Im[\lambda]},
	\end{align*}
	and in the same way also obtain
	\begin{align*}
		\int_{-\infty}^0 \abs{\phi_-}^2 \der[V]{x} = - \abs{c_-(\lambda)}^2 \frac{\Im[m_-]}{\Im[\lambda]}.
	\end{align*}
	Using this, 
	\begin{align*}
		W(\phi_+, \phi_-) = \phi_+(0; \lambda) \phi_-'(0; \lambda) - \phi_+'(0; \lambda) \phi_-(0; \lambda) = c_+(\lambda) c_-(\lambda) \left(m_-(\lambda) - m_+(\lambda)\right)
	\end{align*}
	as well as \cref{lem:wronskian_identity}, we can express \eqref{eq:loc:matrix_formula:1} as 
	\begin{align}\label{eq:loc:matrix_formula:2}
		v_i \overline{v_j} M_{ij}(\lambda)
		&= \frac{\Im[\lambda]}{\pi \abs{W(\phi_+, \phi_-)}^2} \left( \abs{v \cdot v_-(\lambda)}^2 \int_0^\infty \abs{\phi_+}^2 V \der x + \abs{v \cdot v_+(\lambda)}^2 \int_{-\infty}^0 \abs{\phi_-}^2 V \der x \right)
	\end{align}

	\textit{Part 2:}
	For the limit, let $\lambda_0 \in \R \setminus S$ and $\lambda \in \H$.
	Since $\phi_+(x + X^+) = \rho_+(x) \phi_+(x)$ we have
	\begin{align*}
		\int_0^\infty \abs{\phi_+(x; \lambda)}^2 \der[V] x
		= \int_0^{R^+} \abs{\phi_+(x; \lambda)}^2 \der[V] x
		+ \frac{1}{1 - \abs{\rho_+(\lambda)}^2} \int_{R^+}^{R^+ + X^+} \abs{\phi_+(x; \lambda)}^2 \der[V] x
	\end{align*}
	If $\abs{\rho_+(\lambda_0)} = 1$, then by \cref{lem:rho} we have $\rho_+'(\lambda_0) \overline{\rho_+(\lambda_0)} = \ii \abs{\rho_+'(\lambda_0)}$ and thus
	\begin{align*}
		1 - \abs{\rho_+(\lambda)}^2
		= -2 \Re[(\lambda - \lambda_0) \rho_+'(\lambda_0) \overline{\rho_+(\lambda_0)}] + \landauO(\abs{\lambda - \lambda_0}^2)
		= 2 \abs{\rho_+'(\lambda_0)} \Im[\lambda] + \landauO(\abs{\lambda - \lambda_0}^2) 
	\end{align*}
	as $\lambda \to \lambda_0$. It follows that
	\begin{align*}
		\Im[\lambda] \int_0^\infty \abs{\phi_+(x; \lambda)}^2 \der[V] x
		\to \begin{cases}
			\displaystyle \frac{1}{2 \abs{\rho_+'(\lambda_0)}} \int_{R^+}^{R^+ + X^+} \abs{\phi_+(x; \lambda_0)}^2 \der[V] x, & \abs{\rho_+(\lambda_0)} = 1, 
			\\ 0, & \abs{\rho_+(\lambda_0)} < 1
		\end{cases}
	\end{align*}
	as $\lambda \to \lambda_0$ along curves with $\abs{\lambda - \lambda_0}^2 = \landauo(\Im[\lambda])$. As $\lambda_0 \not \in S_0$, we have $W(\phi_+(\impvar; \lambda_0), \phi_-(\impvar; \lambda_0)) \neq 0$. So the remaining terms in the right-hand side of \eqref{eq:loc:matrix_formula:2} are continuous at $\lambda_0$. Thus we so far have shown
	\begin{align}\label{eq:loc:matrix_formula:3}
		\begin{split}
			\lim_{\eps \to 0+} v_i \overline v_j M(\lambda + \ii \eps)
			&= \frac{\bbone_\set{\abs{\rho_+(\lambda)} = 1} \abs{v \cdot v_-(\lambda)}^2 \norm{\phi_+}_{L^2_V(R^+, R^++X^+)}^2}{2 \pi \abs{W(\phi_+, \phi_-)}^2 \abs{\rho_+'(\lambda)}}
			\\ &\quad+ \frac{\bbone_\set{\abs{\rho_-(\lambda)} = 1} \abs{v \cdot v_+(\lambda)}^2 \norm{\phi_-}_{L^2_V(R^--X^-, R^-)}^2}{2 \pi \abs{W(\phi_+, \phi_-)}^2 \abs{\rho_-'(\lambda)}}
		\end{split}
	\end{align}
	for $\lambda \in \R \setminus S$, and one easily sees that this convergence is 
	locally uniform.
	Using \cref{lem:wronskian_identity} we have
	\begin{align*}
		W(\phi_+, \phi_-)
		&= W(\phi_+, r(\phi_-; \phi_+) \phi_+ + t(\phi_-; \phi_+) \overline{\phi_+})
		\\ & = t(\phi_-; \phi_+) W(\phi_+, \overline{\phi_+})
		= t(\phi_-; \phi_+) \frac{\rho_+(\lambda)}{\rho_+'(\lambda)} \norm{\phi_+}_{L^2_V(R^+, R^++X^+)}^2
	\intertext{as well as its analogue}
		W(\phi_+, \phi_-) & = -W(\phi_-,\phi_+)
		= t(\phi_+; \phi_-) \frac{\rho_-(\lambda)}{\rho_-'(\lambda)} \norm{\phi_-}_{L^2_V(R^--X^-, R^-)}^2.
	\end{align*}
	The claim follows from this and \eqref{eq:loc:matrix_formula:3}.
\end{proof}

We complete the description of the spectral measure by considering the point spectrum.

\begin{lemma}\label{lem:functional_calc:point_spectrum}
	Let $\lambda \in \R$. Then $\mu$ is discontinuous at $\lambda$ if and only if $\lambda \in \sigma_p(L)$. In this case, letting $\phi(x; \lambda) = v\cdot \Psi(x; \lambda)$ be an eigenfunction, we have
	\begin{align*}
		\mu(\lambda+) - \mu(\lambda-)
		= \frac{(v_i \overline{v_j})_{i,j=1}^2}{\norm{\phi}_{L^2_V(\R)}^2} .
	\end{align*}
	Note that $v \in \C v_+(\lambda) = \C v_-(\lambda)$ holds for $\lambda \in \sigma_p(L)$, and in particular $\sigma_p(L) \subseteq S_0$.
\end{lemma}

\begin{proof}
	We write $\Delta \mu(\lambda) = \mu(\lambda+) - \mu(\lambda-)$.

	\textit{Part 1:}
	Let $\mu$ be discontinuous at $\lambda_0$. For $w \in \C^2$ we have
	\begin{align*}
		f = T^{-1}[\bbone_\set{\lambda_0} w]
		= w_i \Psi_j(\impvar; \lambda_0) \Delta \mu_{ij}(\lambda_0) \in L^2_V(\R).
	\end{align*}
	\cref{lem:calculus_diagonalizes_L0} yields $f \in H^2(\R)$ with $L f = \lambda_0 f$. Since there is at most $1$ linearly independent eigenfunction $\phi$, $\Delta \mu(\lambda_0)$ has rank $1$. As $\Delta \mu(\lambda_0)$ is positive semidefinite and $w_i \Delta \mu_{ij}(\lambda_0) e_j \in \C v$ there exists $r > 0$ such that $\Delta \mu_{ij}(\lambda_0) = r \overline{v_i} v_j$ holds. It remains to check the value of $r$:
	\begin{align*}
		r \abs{v}^4 = v_i \overline{v_j} \Delta\mu_{ij}(\lambda_0)
		= \norm{\bbone_\set{\lambda_0} v}_{L^2(\mu)}^2
		= \norm{T^{-1}[\bbone_\set{\lambda_0} v]}_{L^2_V}^2
		= \norm{r \abs{v}^2 \phi}_{L^2_V}^2
		= r^2 \abs{v}^4 \norm{\phi}_{L^2_V}^2.
	\end{align*}
	Note that $\phi$ is a multiple of a real-valued function since it is an eigenfunction to a simple real eigenvalue.
	Hence $v$ is a multiple of a real vector and in particular $\overline{v_i} v_j = v_i \overline{v_j}$ holds.
	
	\textit{Part 2:} Now let $\lambda_0 \in \sigma_p(L)$ with eigenfunction $\phi$. Then $\lambda_0 T[\phi] (\lambda) = T[L \phi] (\lambda) = \lambda T[\phi] (\lambda)$ for $\lambda \in \R$ by \cref{lem:calculus_diagonalizes_L0} and therefore $T[\phi] = \bbone_\set{\lambda_0} w$ for some $w \in \C^2$. As
	\begin{align*}
		0 < \norm{\phi}_{L^2_V}^2 = \norm{T[\phi]}_{L^2(\mu)}^2 = w_i \overline{w_j} \Delta\rho_{ij}(\lambda_0),
	\end{align*}
	$\mu$ must be discontinuous at $\lambda_0$.
\end{proof}

\begin{notation}
	In the following, for $\lambda \in \sigma_p(L)$ we denote by $\phi_0$ and $v_0$ the eigenfunction and eigenvector from \cref{lem:functional_calc:point_spectrum}.
\end{notation}

Let us conclude this subsection by explicitly writing down the $L^2(\mu)$-norm.

\begin{theorem}\label{thm:spectral_measure:explicit}
	Let $f \colon \R \to \C^2$ be measurable. Then
	\begin{align*}
		\norm{f}_{L^2(\mu)}^2
		&= \frac{1}{2 \pi} \int_{\R\setminus S} \frac{\bbone_\set{\abs{\rho_+(\lambda)} = 1} \abs{f(\lambda) \cdot v_-(\lambda)}^2 \abs{\rho_+'(\lambda)}}{\abs{t(\phi_-; \phi_+)}^2 \norm{\phi_+}_{L^2_V(R^+, R^+ + X^+)}^2} 
		+ \frac{\bbone_\set{\abs{\rho_-(\lambda)} = 1} \abs{f(\lambda) \cdot v_+(\lambda)}^2 \abs{\rho_-'(\lambda)}}{\abs{t(\phi_+; \phi_-)}^2 \norm{\phi_-}_{L^2_V(R^--X^-, R^-)}^2} \der \lambda
		\\ &+ \sum_{\lambda \in \sigma_p(L)} \frac{\abs{f(\lambda) \cdot v_0(\lambda)}^2}{\norm{\phi_0}_{L^2_V(\R)}^2}.
	\end{align*}
\end{theorem}

\begin{proof}
	By \cref{lem:rho,rem:S0} the set $S$ is at most countable and its complement has at most countably many connected components. Let us write\footnote{For simplicity of notation, we assume that both sets are countably infinite.} $S = \set{\lambda_n \colon n \in \N}$, $\R\setminus S = \cup_{n \in \N} I_n$. It follows that
	\begin{align*}
		\norm{f}_{L^2(\mu)}^2
		= \sum_{n \in \N} \norm{f \bbone_{I_n}}_{L^2(\mu)}^2
		+ \sum_{n \in \N} \norm{f(\lambda_n) \bbone_\set{\lambda_n}}_{L^2(\mu)}^2.
	\end{align*}
	\cref{lem:functional_calc:point_spectrum} shows that the second series is equal to
	\begin{align*}
	    \sum_{\lambda \in \sigma_p(L)} \frac{\abs{f(\lambda) \cdot v_0(\lambda)}^2}{\norm{\phi_0}_{L^2_V(\R)}^2}  
    \end{align*}

	For the first series, let $n \in \N$ and $[a, b] \subseteq I_n$. By \cref{thm:functional_calculus:general,prop:spectral_density} we have 
	\begin{align*}
		\int_a^b M(\lambda) \der \lambda
		= \lim_{\eps \to 0+} \int_a^b M(\lambda + \ii \eps) \der \lambda
		= \mu(b) - \mu(a).
	\end{align*}
	As $a, b$ were arbitrary, it follows that
	\begin{align*}
		\int_{I_n} f_i(\lambda) \overline{f_j(\lambda)} M_{ij}(\lambda) \der \lambda
		= \int_{I_n} f_i(\lambda) \overline{f_j(\lambda)} \der \mu_{ij}(\lambda)
	\end{align*}
	holds. Summing over $n$ and using the formula for $M(\lambda)$ given in \cref{prop:spectral_density} completes the proof.
\end{proof}

\subsection{Embeddings}\label{subsec:embeddings}

Let us generalize the torus $\T$ to a measure space $\Omega$ and $-\partial_t^2$ to a formal symmetric operator $\calL$ on $\Omega$ satisfying $\calL e_k = \nu_k e_k$, where  $(e_k)_{k\in\N}$ is an orthonormal system in $L^2(\Omega)$ and $\nu_k \geq 0$. For $k \in \N$, set $\hat u_k(x) = \int_\Omega u(x, t) \overline{e_k(t)} \der t$.
We consider the differential operator $L - \calL$ on $\R \times \Omega$, with $L$ acting on $x \in \R$ and $\calL$ on $t \in \Omega$. We associate to it the sesquilinear form
\begin{align*}
	b(u, v) 
	\coloneqq \int_{\R \times \Omega} \bigl((L - \calL) u \cdot \overline{v} \bigr) \der[V(x)](x, t)
	\coloneqq 
	\sum_{k \in \N} \int_\R \bigl( (\lambda - \nu_k) T_i[\hat u_k](\lambda) \cdot \overline{T_j[\hat v_k] (\lambda)} \bigr) \der \mu_{ij}(\lambda)
\end{align*}
Next we give the domain of the above form.
\begin{definition}
	Using the closed subspace
	\begin{align*}
		L^2_{(e_k)}
		&\coloneqq \Set{u \in L^2_V(\R \times \Omega) \colon u(x, \impvar) \in \overline{\lspan\!\set{e_k \colon k \in \N}}}
		\\ &= \Set{u \in L^2_V(\R \times \Omega) \colon u(x, t) = \sum_{k \in \N} \hat u_k(x) e_k(t)}
	\end{align*}
	of $L^2_V(\R \times \Omega)$, we define the domain $\calH$ of $b$ by
	\begin{align*}
		\calH = \set{u \in L^2_{(e_k)} \colon \norm{u}_\calH^2 = \ip{u}{u}_\calH < \infty}
	\end{align*}
	where 
	\begin{align*}
		\ip{u}{v}_\calH \coloneqq \sum_{k \in \N} \int_\R \bigl( \abs{\lambda - \nu_k} T_i[\hat u_k] (\lambda)\cdot \overline{T_j[\hat v_k] (\lambda)} \bigr) \der \mu_{ij}(\lambda).
	\end{align*}
\end{definition}

Our main goal is to investigate embeddings $\calH \embeds L^p_V(\R \times \Omega)$. We use interpolation in $p$. The end point $p=2$ will be trivial, and we prepare $p=\infty$ by first showing $L^\infty$-bounds for the spectral density as well as eigenfunctions.

\begin{lemma}\label{lem:applied_uniform_bound}
	There exists a constant $C > 0$ (independent of $\lambda$) such that
	\begin{align*}
		\frac{\norm{\phi_-}_{L^\infty(\R)}}{\abs{t(\phi_-; \phi_+)} \norm{\phi_+}_{L^2_V(R^+, R^++X^+)}} \leq C
	\end{align*}
	holds for all $\lambda \in \R\setminus S$ with $\rho_+(\lambda) \in \SwR$, analogously
	\begin{align*}
		\frac{\norm{\phi_+}_{L^\infty(\R)}}{\abs{t(\phi_+; \phi_-)} \norm{\phi_-}_{L^2_V(R^--X^-, R^-)}} \leq C
	\end{align*}
	holds for $\lambda \in \R \setminus S$ with $\rho_-(\lambda) \in \SwR$, and lastly
	\begin{align*}
		\frac{\norm{\phi_0}_{L^\infty(\R)}}{\norm{\phi_0}_{L^2_V(\R)}} \leq C
	\end{align*}
	holds for $\lambda \in \sigma_p(L)$.
\end{lemma}

\begin{proof}
	We consider the first inequality, and calculate
	\begin{align*}
		W(\phi_-, \overline{\phi_-})
		= W(r \phi_+ + t \overline{\phi_+}, \overline{r} \overline{\phi_+} + \overline{t} \phi_+) 
		= (\abs{r}^2 - \abs{t}^2) W(\phi_+, \overline{\phi_+}).
	\end{align*}
	Let us show that $\abs{r(\phi_-; \phi_+)} \leq \abs{t(\phi_-; \phi_+)}$. For this, by \cref{lem:wronskian_identity,lem:rho} we first have
	\begin{align*}
		W(\phi_+, \overline{\phi_+}) = \frac{\rho_+(\lambda)}{\rho_+'(\lambda)} \norm{\phi_+}_{L^2_V(R^+,R^++X^+)^2} \in \ii (-\infty, 0).
	\end{align*}
	\textit{Case 1:} Assume $\rho_-(\lambda) \in \SwR$. Then these lemmas also show $W(\phi_-, \overline{\phi_-}) \in \ii (0, \infty)$, so we even have $\abs{r} < \abs{t}$.

	\textit{Case 2:} Assume $\abs{\rho_-(\lambda)} < 1$. Recall that $\phi_-(x - X^-) = \rho_-(\lambda) \phi_-(x)$ holds. As this equation has at most $1$ linearly independent solution and $\rho_-(\lambda) \in \R$, we see that $\phi_-$ is a real-valued function up to multiplication with a complex scalar. Therefore $W(\phi_-, \overline{\phi_-}) = 0$, and in particular $\abs{r} = \abs{t}$. 
	
	Thus $\abs{r} \leq \abs{t}$ holds. Using semiperiodicity of both $\phi_+$ and $\phi_-$ in their respective regions as well as \cref{prop:uniform_ef_bound}, we then estimate
	\begin{align*}
		\norm{\phi_-}_\infty 
		&= \norm{\phi_-}_{L^\infty(R^--X^-, \infty)}
		= \norm{r \phi_+ + t \overline{\phi_+}}_{L^\infty(R^--X^-, \infty)}
		\\ &\leq 2 \abs{t} \norm{\phi_+}_{L^\infty(R^--X^-, \infty)}
		= 2 \abs{t} \norm{\phi_+}_{L^\infty(R^--X^-, R^++X^+)}
		\lesssim \abs{t} \norm{\phi_+}_{L^2(R^+, R^+ + X^+)}
	\end{align*}
	uniformly in $\lambda$.

	The second inequality can be shown in analogy to the first. So let us consider the third inequality. Here, using semiperiodicity of $\phi_0$ on both $(-\infty, R^-)$ and $(R^+, \infty)$ as well as \cref{prop:uniform_ef_bound} we find
	\begin{align*}
		\norm{\phi_0}_\infty 
		= \norm{\phi_0}_{L^\infty(R^- - X^-, R^+ + X^+)}
		\lesssim \norm{\phi_0}_{L^2_V(R^- - X^-, R^+ + X^+)}
		\leq \norm{\phi_0}_{L^2_V(\R)}
		&\qedhere
	\end{align*}
\end{proof}

Using the bounds of \cref{lem:applied_uniform_bound} we obtain a sufficient condition for boundedness of $L^p$-embeddings.

\begin{lemma}\label{lem:Lp-embedding}
	Let $p \in (2,\infty]$, let $(I_n^\pm)_{n\in\N}$ enumerate the connected components of $\R\setminus S_\pm$ on which $\abs{\rho_\pm} = 1$ holds, and assume that
	\begin{align}\label{eq:loc:condition_for_Lp}
		C \coloneqq \sum_{k\in\N} \norm{e_k}_\infty^2  
		\left( \sum_{n \in \N} \dist(\nu_k, I_n^+)^{-s} + \sum_{n \in \N} \dist(\nu_k, I_n^-)^{-s} + \sum_{\lambda \in \sigma_p(L)} \abs{\nu_k - \lambda}^{-s} \right)
		 < \infty
	\end{align}
	where $s \coloneqq \frac{p}{p-2}$.
	Then the embedding $\calH \embeds L^p_V(\R \times \Omega)$ is bounded.
\end{lemma}
\begin{proof}
	Observe that the map
	\begin{align*}
		E \colon \calH \to L^2_{(e_k)}, 
		~ 
		u \mapsto \sum_{k \in \N} e_k(t) T^{-1}\bigl[\abs{\lambda - \nu_k}^{\frac12} T[\hat u_k](\lambda)\bigr](x)
	\end{align*} 
	is an isometric isomorphism. Set $m_k(\lambda) = \abs{\lambda - \nu_k}^{-\frac{s}{2}}$ and for $\theta \in [0, 1]$, $q \coloneqq \frac{2}{1 - \theta}$ we consider the map 
	\begin{align*}
		\iota_\theta \colon L^2_{(e_k)} \to L^q_V(\R \times \Omega),
		\quad
		u \mapsto \sum_{k \in \N} e_k(t) T^{-1}[m_k^\theta T[\hat u_k](\lambda)](x). 
	\end{align*}
    Our goal is to show that $\iota_\theta\circ E: \calH \to L^q_V(\R\times\Omega)$ is bounded for $\theta\in [0,1]$ and $q=\frac{2}{1-\theta}$ using interpolation. If we then set $\theta=\frac{1}{s}$ we get $p=q$, $\iota_\theta\circ E=\Id$ and thus the proof will be finished.

    \medskip
    
	\textit{Part 1:} First, let $\theta = 1$, $q = \infty$ and assume that $u \in L^2_{(e_k)}$ has finitely many nonzero modes $\hat u_k$ and that $T[\hat u_k]$ is compactly supported for all $k$. Then we calculate 
	\begin{align*}
		\abs{\iota_1[u](x, t)} 
		&= \abs{\sum_{k \in \N} e_k(t) \int_\R m_k(\lambda) T_i[\hat u_k](\lambda) \Psi_j(x; \lambda) \der \mu_{ij}(\lambda)}
		\\ &\leq \left( \sum_{k \in \N} \norm{T[\hat u_k]}_{L^2(\mu)}^2 \right)^{\nicefrac12}
		\left( \sum_{k \in \N} \abs{e_k(t)}^2 \norm{m_k (\impvar) \Psi(x; \impvar)}_{L^2(\mu)}^2 \right)^{\nicefrac12} 
	\end{align*}
	Note that the first term is
	\begin{align*}
		\sum_{k \in \N} \norm{T[\hat u_k]}_{L^2(\mu)}^2
		= \sum_{k \in \N} \norm{\hat u_k}_{L^2_V(\R)}^2 
		= \norm{u}_{L^2_{(e_k)}}^2.
	\end{align*}
	Now we consider the second term. Using the identities $\phi_\pm(x; \lambda) = \Psi(x; \lambda) \cdot v_\pm(\lambda), \phi_0(x; \lambda) = \Psi(x; \lambda) \cdot v(\lambda)$ as well as \cref{thm:spectral_measure:explicit,lem:applied_uniform_bound} we obtain
	\begin{align*}
		\MoveEqLeft \norm{m_k(\impvar) \Psi(x; \impvar)}_{L^2(\mu)}^2
		\\ &= \frac{1}{2 \pi} \int_{\R\setminus S} \frac{\bbone_\set{\abs{\rho_+(\lambda)} = 1} \abs{m_k(\lambda)}^2 \abs{\phi_-(x)}^2 \abs{\rho_+'(\lambda)}}{\abs{t(\phi_-; \phi_+)}^2 \norm{\phi_+}_{L^2_V(R^+, R^+ + X^+)}^2} 
		+ \frac{\bbone_\set{\abs{\rho_-(\lambda)} = 1} \abs{m_k(\lambda)}^2 \abs{\phi_+(x)}^2 \abs{\rho_-'(\lambda)}}{\abs{t(\phi_+; \phi_-)}^2 \norm{\phi_-}_{L^2_V(R^--X^-, R^-)}^2} \der \lambda
		\\ &\quad+ \sum_{\lambda \in \sigma_p(L)} \frac{\abs{m_k(\lambda)}^2 \abs{\phi_0(x)}^2}{\norm{\phi_0}_{L^2_V(\R)}^2}
		\\ &\lesssim \int_{\R\setminus S} \abs{m_k(\lambda)}^2 \left(\bbone_\set{\abs{\rho_+(\lambda)} = 1} \abs{\rho_+'(\lambda)} + \bbone_\set{\abs{\rho_-(\lambda)} = 1} \abs{\rho_-'(\lambda)} \right) \der \lambda
		+ \sum_{\lambda \in \sigma_p(L)} \abs{m_k(\lambda)}^2
	\end{align*}
	uniformly in $k, x$. Using \cref{lem:rho} we can further estimate
	\begin{align*}
		\int_{\R \setminus S} \bbone_\set{\abs{\rho_+(\lambda)} = 1}  \abs{m_k(\lambda)}^2 \abs{\rho_+'(\lambda)} \der \lambda
		&= \sum_{n \in \N} \int_{I_n^+} \abs{m_k(\lambda)}^2 \abs{\rho_+'(\lambda)} \der \lambda \\
		& \leq \sum_{n \in \N} \dist(\nu_k, I_n^+)^{-s} \int_{I_n^+} \abs{\rho_+'(\lambda)} \der \lambda
		\\ &\leq \pi \sum_{n \in \N} \dist(\nu_k, I_n^+)^{-s},
	\end{align*}
	with similar estimates for the other two terms. Recalling the definition of $C$, we so far have shown
	\begin{align}\label{eq:loc:Linf_estimate}
		\norm{\iota_1[u]}_\infty \lesssim C^{\frac12} \norm{u}_{L^2_{(e_k)}}
	\end{align}
	for $u \in L^2_{(e_k)}$ with compact frequency support. By density, \eqref{eq:loc:Linf_estimate} holds for all $u \in L^2_{(e_k)}$.

	\textit{Part 2:} For $\theta = 0$ and $q = 2$, $\iota_0 \colon L^2_{(e_k)} \to L^2_V$ is the identity map. By interpolation, cf. e.g. \cite[Chapter~2]{Lunardi}, the map $\iota_\theta \colon L^2_{(e_k)} \to L^q_V(\R \times \Omega)$, $q = \frac{2}{1 - \theta}$, is bounded for all $\theta \in [0, 1]$. The claim of \cref{lem:Lp-embedding} follows by setting $\theta = \frac{1}{s}$ since for this choice $\iota_\theta E$ is the identity map and $p = q$ holds.
\end{proof}

\begin{theorem}\label{thm:Lp-embedding:general}
	Let $N \in \N$ and assume for the eigenvalues $\nu_k$ of $\calL$ that $\# \set{k \colon \sqrt{\nu_k} \in B} \leq N$ for any interval $B$ of length $1$, which implies that $(\nu_k)$ grows at least quadratically. Generalizing \ref{ass:spectrum} we assume additionally
	\begin{align*}
		\inf\set{\abs{\sqrt{\lambda} - \sqrt{\nu_k}} \colon k \in \N, \lambda \in \sigma(L)} > 0.
	\end{align*}
	Moreover, assume $\alpha, \beta \geq 0$ exist such that the eigenfunctions of $\calL$ satisfy $\norm{e_k}_\infty \lesssim \nu_k^{\frac\alpha2}$, and generalizing \ref{ass:point_spectrum} assume for the point spectrum of $L$ that
	\begin{align*}
		\sum_{\lambda \in \sigma_p(L)} \lambda^{-r} < \infty
	\end{align*}
	holds for $r > \frac12 + \beta$. Then the embedding $\calH \embeds L^p_V(\R \times \Omega)$ is bounded and $\calH \embeds L^p_V(A \times \Omega)$ is compact for all $p \in [2, 2 + \frac{1}{\alpha + \beta})$ and $A \subseteq \R$ compact.
\end{theorem}
\begin{proof} 
	Set $\delta \coloneqq \inf\set{\abs{\sqrt{\lambda} - \sqrt{\nu_k}} \colon k \in \N, \lambda \in \sigma(L)}$ and $s \coloneqq \frac{p}{p - 2} > 1 + 2\alpha + 2 \beta$. We show that the assumptions of \cref{lem:Lp-embedding} are satisfied. 

	\textit{Part 1:} We estimate the spectral bands $I_n^+$. By definition, for $\lambda \in \set{\min I_n^+, \max I_n^+}$ we have $\rho_+(\lambda) \in \set{-1, 1}$. Therefore there exists an $X$-periodic or $X$-antiperiodic solution $\phi$ of the problem
	\begin{align} \label{eq:periodic_plus}
		- \frac{1}{V_\per^+(x)} \dxsquare \phi = \lambda \phi, 
		\qquad
		\phi(x + 2 X) = \phi(x).
	\end{align}
	If we enumerate the eigenvalues of \eqref{eq:periodic_plus} in increasing order by $(\lambda_j^+)_{j \in \N}$, we get in particular
	\begin{align*}
		\sum_{n \in \N} \dist(\nu_k, I_n^+)^{-s}
		\leq 2 \sum_{j \in \N} \abs{\nu_k - \lambda_j^+}^{-s}.
	\end{align*}
	Moreover, using that the eigenvalues of the $2X$-periodic Laplacian are $0, \frac{\pi}{X}, \frac{\pi}{X}, 2 \frac{\pi}{X}, 2 \frac{\pi}{X}, \dots$ (with eigenfunctions $1, \sin(\frac{\pi x}{X}), \cos(\frac{\pi x}{X}), \sin(2 \frac{\pi x}{X}), \cos(2 \frac{\pi x}{X}), \dots$), by the Min-Max-Principle (cf. \cite[Chapter~11.2, Theorem~1]{straussPDE}) we have
	\begin{align*}
		\lambda_j^+
		&= \inf_{\substack{Y \subseteq H^1_\per(\R) \\ \dim(Y) = j}} 
		~\sup_{\substack{u \in Y\\ u \neq 0}} \frac{\int_0^{2X} \abs{u'}^2 \der x}{\int_0^{2X} \abs{u}^2 \der[V_\per^+(x)]{x}}
		\\ &\geq \frac{1}{\norm{V_\per^+}_\infty} \inf_{\substack{Y \subseteq H^1_\per(\R) \\ \dim(Y) = j}} 
		~\sup_{\substack{u \in Y\\ u \neq 0}} \frac{\int_0^{2X} \abs{u'}^2 \der x}{\int_0^{2X} \abs{u}^2 \der{x}}
		= \frac{1}{\norm{V_\per^+}_\infty} \left(\frac{\pi \floor{\frac{j}{2}}}{X}\right)^2.
	\end{align*}

	\textit{Part 2:} 
	The following estimates are inspired by \cite{maier_reichel_schneider}. First, we estimate the double sum via
	\begin{align*}
		\sum_{k \in \N} \norm{e_k}_\infty^2 \sum_{n \in \N} \dist(\nu_k, I_n^+)^{-s}
		&\lesssim \sum_{k \in \N} \abs{\nu_k}^\alpha \sum_{j \in \N} \abs{\nu_k - \lambda_j^+}^{-s}\\
		& = \sum_{k \in \N} \abs{\nu_k}^\alpha \sum_{j \in \N} \Abs{\sqrt{\nu_k} - \sqrt{\lambda_j^+}}^{-s} \Abs{\sqrt{\nu_k} + \sqrt{\lambda_j^+}}^{-s}
		\\ &\leq \sum_{j \in \N} \max\Set{\sqrt{\lambda_j^+}, \delta}^{-s} \sum_{k \in \N} \abs{\nu_k}^\alpha \Abs{\sqrt{\nu_k} - \sqrt{\lambda_j^+}}^{-s}.
	\end{align*}
    Next, observe that for each $k\in \N$ we have $\sqrt{\nu_k}\in \bigcup_{m=0}^\infty [m,m+1)$ as well as $\sqrt{\nu_k}-\sqrt{\lambda_j^+}\in \bigcup_{m=-\infty}^\infty [m,m+1)$. Together with the assumption on the number of eigenvalues per interval of length $1$, we use this (combined with a separate consideration for $m=-1,0$) to estimate for fixed $j$
	\begin{align*}
		\sum_{\substack{k \in \N \\ \abs{\nu_k} \leq 2 \lambda_j^+}} \abs{\nu_k}^\alpha \Abs{\sqrt{\nu_k} - \sqrt{\lambda_j^+}}^{-s}
		&\leq 2^\alpha (\lambda_j^+)^\alpha \sum_{k \in \N} \abs{\sqrt{\nu_k} - \sqrt{\lambda_j^+}}^{-s}
		\leq 2^{\alpha} (\lambda_j^+)^\alpha \cdot 2 N \left(\delta^{-s} + \sum_{m=1}^{\infty} m^{-s} \right),
		\\
		\sum_{\substack{k \in \N \\ \abs{\nu_k} > 2 \lambda_j^+}} \abs{\nu_k}^\alpha \Abs{\sqrt{\nu_k} - \sqrt{\lambda_j^+}}^{-s}
		&\leq \left(1 - \tfrac{1}{\sqrt{2}}\right)^{-s} \sum_{k \in \N} \abs{\sqrt{\nu_k}}^{2 \alpha - s}
		\leq \left(1 - \tfrac{1}{\sqrt{2}}\right)^{-s} 2 N \left(\delta^{2 \alpha - s} + \sum_{m=1}^{\infty} m^{2 \alpha - s} \right).
	\end{align*}
	Therefore we obtain
	\begin{align*}
		\sum_{j \in \N} \max\Set{\sqrt{\lambda_j^+}, \delta}^{-s} \sum_{k \in \N} \abs{\nu_k}^\alpha \Abs{\sqrt{\nu_k} - \sqrt{\lambda_j^+}}^{-s}
		\leq C \sum_{j\in\N} \max\Set{\sqrt{\lambda_j^+}, \delta}^{-s} (1 + (\lambda_j^+)^\alpha)
	\end{align*}
	where the right-hand side is finite since $\lambda_j^+ \gtrsim (j-1)^2$ by part 1 and $s - 2 \alpha > 1$ by assumption.

	\textit{Part 3:}
	So far, we have shown
	\begin{align*}
		\sum_{k \in \N} \norm{e_k}_\infty^2 \sum_{n \in \N} \dist(\nu_k, I_n^+)^{-s} < \infty.
	\end{align*}
	By the arguments above, this also holds for $I_n^-$ and we have the estimate
	\begin{align*}
		\sum_{k \in \N} \norm{e_k}_\infty^2 \sum_{\lambda \in \sigma_p(L)} \abs{\nu_k - \lambda}^{-s}
		\leq C \sum_{\lambda \in \sigma_p(L)} \max\Set{\sqrt{\lambda}, \delta}^{-s} (1 + \lambda^\alpha)
	\end{align*}
	which is finite by assumption.

	\textit{Part 4:} It remains to show local compactness of the embedding $E \colon \calH \to L^p(\R \times \Omega)$. We only consider $p > 2$ and take $A \subseteq \R$ compact. For $K \in \N$ consider
	\begin{align*}
		E_K \colon \calH \to L^p_V(\R \times \Omega), 
		u \mapsto \sum_{k=1}^{K} e_k(t) \hat u_k(x).
	\end{align*}
	As the map $\calH \to H^1(\R), u \mapsto \hat u_k$ is bounded for each $k$, $E_K$ is a compact operator. A small modification of \cref{lem:Lp-embedding} shows that $\norm{E - E_K} \lesssim C_K^{\frac{1}{2 s}}$ where 
	\begin{align*}
		C_K \coloneqq \sum_{\substack{k \in \N \\ k > K}} \norm{e_k}_\infty^2  
		\left( \sum_{n \in \N} \dist(\nu_k, I_n^+)^{-s} + \sum_{n \in \N} \dist(\nu_k, I_n^+)^{-s} + \sum_{\lambda \in \sigma_p(L)} \abs{\nu_k - \lambda}^{-s} \right)
	\end{align*}
	As the series in \eqref{eq:loc:condition_for_Lp} converges, $C_K \to 0$ as $K \to \infty$, and therefore the limit $E \colon \calH \to L^p_V(A \times \Omega)$ is also compact.
\end{proof}

%% file: vector_L2.tex

\section{Vector-valued \texorpdfstring{$L^2$}{L2}-spaces}
\label{sec:vector_L2}

Following \cite{dunfordschwartz}, we give a definition of the Hilbert space $L^2(\mu)$ appearing in \cref{sec:embeddings} where $\mu$ is an increasing matrix-valued function.

\begin{definition}
	We call a function $\mu \colon \R \to \C^{d \times d}$ \emph{increasing} if $\mu(y) - \mu(x)$ is Hermitian and positive semidefinite for all $y \geq x$.
\end{definition}

\begin{lemma} \label{lem:rs_integral}
	Let $\mu \colon \R \to \C^{d \times d}$ be increasing. Then the coefficients $\mu_{ij}$ are locally of bounded variation, and hence there exist $\C$-valued measures $\nu_{ij}$, defined on the bounded Borel subsets of $\R$, which are $\sigma$-additive for sets with bounded union, such that
	\begin{align*}
		\int_\R \varphi \der \mu_{ij}
		= \int_\R \varphi \der \nu_{ij}
	\end{align*}
	for all $\varphi \in C_c(\R)$. Here, the left-hand side is a Riemann-Stieltjes integral.
\end{lemma}
\begin{proof}
	By positive definiteness, for $y \geq x$ we have
	\begin{align*}
		\abs{\mu_{ij}(y) - \mu_{ij}(x)}
		\leq \sqrt{\left(\mu_{ii}(y) - \mu_{ii}(x)\right)
		\left(\mu_{jj}(y) - \mu_{jj}(x)\right)}
		\leq \tfrac12 \left( \mu_{ii}(y) - \mu_{ii}(x) + \mu_{jj}(y) - \mu_{jj}(x) \right).
	\end{align*}
	Since $\mu_{ii}$ and $\mu_{jj}$ are increasing, it follows that
	\begin{align*}
		\norm{\mu_{ij}}_{BV[a,b]}
		\leq \tfrac12 \left( \mu_{ii}(b) - \mu_{ii}(a) + \mu_{jj}(b) - \mu_{jj}(a) \right) < \infty.
	\end{align*}
	Existence of the measure $\mu_{ij}$ then follows from the Riesz-Markov-Kakutani representation theorem, see e.g. \cite[Theorem~6.19]{rudin}
\end{proof}

\begin{definition}
	Let $\mu \colon \R \to \C^{d \times d}$ be increasing, $\nu_{ij}$ be the measures from \cref{lem:rs_integral}, and $f \colon \R \to \C^d$ be Borel measurable. Let $\nu$ be a $\sigma$-finite Borel measure on $\R$ such that all $\nu_{ij}$ are absolutely continuous with respect to $\nu$. Then define
	\begin{align*}
		\norm{f}_{L^2(\mu)}^2
		\coloneqq \int_\R f_i \overline{f_j} \der \mu_{ij}(\lambda)
		\coloneqq \int_\R \left( f_i \overline{f_j} \dv{\nu_{ij}}{\nu} \right) \der \nu,
	\end{align*}
	where we used Einstein summation convention and $\dv{\nu_{ij}}{\nu}$ is the Radon-Nikodým derivative.
\end{definition}
\begin{remark}
	By \cite[Lemma~XIII.5.7]{dunfordschwartz}, the matrix $\bigl(\dv{\nu_{ij}}{\nu}\bigr)_{i,j}$ is positive semidefinite $\nu$-almost everywhere. Hence, the last integrand above is nonnegative and therefore the integral exists in $[0, \infty]$. Note that such $\nu$ always exists and that the $L^2(\mu)$-norm does not depend on the choice of $\nu$. 
	One can for example take $\nu(E) \coloneqq \sup_{n \in \N} \sum_{i,j=1}^{d} \abs{\nu_{ij}}(E \cap [-n, n])$. 
\end{remark}
\begin{definition}
	Define $L^2(\mu)$ as the quotient space of $\Set{f \colon \R \to \C^d \text{ meas.} \big\vert \norm{f}_{L^2(\mu)} < \infty}$ modulo $\Set{f \colon \R \to \C^d \text{ meas.} \big\vert \norm{f}_{L^2(\mu)} = 0}$.
	By \cite[Theorem~XIII.5.10]{dunfordschwartz}, $L^2(\mu)$ is a Hilbert space with inner product
	\begin{align*}
		\ip{f}{g}_{L^2(\mu)}
		\coloneqq \int_\R f_i \overline{g_j} \der \mu_{ij}(\lambda)
		\coloneqq \int_\R \left( f_i \overline{g_j} \dv{\nu_{ij}}{\nu} \right) \der \nu.
	\end{align*}
\end{definition}

\begin{remark}
	Multiplication with matrix-valued functions need not be well-defined on $L^2(\mu)$. Consider for example
	\begin{align*}
		\mu(\lambda) = \begin{pmatrix} \lambda & -\lambda \\ -\lambda & \lambda \end{pmatrix} \in \C^{2 \times 2}, \qquad 
		M = \begin{pmatrix}
			1 & 0 \\
			0 & -1
		\end{pmatrix},
	\end{align*}
	so that $\norm{f}_{L^2(\mu)}^2 = \int_\R \abs{f_1 - f_2}^2 \der \lambda$. In particular, $(g,g)^\top = 0$ in $L^2(\mu)$ for arbitrary measurable $g$, whereas $\norm{M (g,g)^\top}_{L^2(\mu)} = 2 \norm{g}_{L^2(\R)}$ need not be zero, nor finite.
	On the other hand, multiplication with scalar-valued measurable functions $m \colon \R \to \C$ is well-defined, and in addition $\norm{m f}_{L^2(\mu)} \leq \sup_{x \in \R}\abs{m(x)} \cdot \norm{f}_{L^2(\mu)}$ holds.
\end{remark}

%% file: eigenfunction_bounds.tex
\section{Eigenfunction bounds}
\label{sec:eigenfunction_bounds}

In this section we discuss uniform estimates on functions $u$ solving 
\begin{align}\label{eq:evp_bounded}
	-u'' = \lambda V(x) u,
	\qquad\text{for}\quad
	x \in I
\end{align}
on a compact interval $I$ and with $\lambda, V$ positive.
For Schrödinger operators, uniform eigenfunction bounds with respect to $\lambda$ of the type $\norm{u}_\infty \lesssim \norm{u}_2$ are known, cf. e.g. \cite{komornik}. 
These can be transferred to the weighted eigenvalue problem \eqref{eq:evp_bounded} using the Liouville transform if $V$ is twice differentiable. However, here we show a generalization of this inequality under the weaker assumption that $V$ is of bounded variation.

\begin{proposition}\label{prop:uniform_ef_bound}
	Let $I, J$ be bounded intervals of positive length with $J \subseteq I$, and $V \in BV(I)$ with $\essinf_I V > 0$. Then there exists a constant $C = C(I, J, V)> 0$ such that 
	\begin{align}\label{eq:loc:uniform_ef_bound}
		\norm{u}_{L^\infty(I)} \leq C \norm{u}_{L^2(J)}	
	\end{align}
	for all $\lambda \geq 0$ and all solutions $u$ to \eqref{eq:evp_bounded}.

	Note that the reverse inequality $\norm{u}_{L^2(J)} \leq \norm{u}_{L^2(I)} \leq \abs{I}^{\nicefrac12} \norm{u}_{L^\infty(I)}$ always holds. 
\end{proposition}
\begin{proof}
	\textit{Part 1:} 
	For fixed $\lambda \in [0, \infty)$, inequality \eqref{eq:loc:uniform_ef_bound} holds since the space of solutions to \eqref{eq:evp_bounded} has dimension $2 < \infty$. As the space of solutions to \eqref{eq:evp_bounded} depends continuously on $\lambda$, so does the optimal constant $C$ in \eqref{eq:loc:uniform_ef_bound}. Therefore it suffices to show that $C$ can be bounded uniformly as $\lambda \to \infty$.

	\textit{Part 2:} 
	We use a change of coordinates similar to the Liouville transform, replacing $(u, u')$ by $(\phi_1, \phi_2)$. For this, let $a \coloneqq \inf I, \mu \coloneqq \sqrt{\lambda}$, and for $ x \in I$ set $t(x) \coloneqq \int_{a}^x \sqrt{V(s)} \der s$ as well as
	\begin{align*}
		\phi(x) \coloneqq
		\begin{pmatrix}
			\cos(\mu t(x)) & - \sin(\mu t(x))
			\\ \sin(\mu t(x)) & \cos(\mu t(x))
		\end{pmatrix} \begin{pmatrix}
			u(x) 
			\\ \frac{1}{\mu \sqrt{V(x)}} u'(x)
		\end{pmatrix},
	\end{align*}
	or equivalently
	\begin{align*}
		\begin{pmatrix}
			u(x) \\ u'(x)
		\end{pmatrix}
		= 
		\begin{pmatrix}
			\cos(\mu t(x)) & \sin(\mu t(x))
			\\ - \mu \sqrt{V(x)} \sin(\mu t(x)) & \mu \sqrt{V(x)} \cos(\mu t(x))
		\end{pmatrix} \phi(x). 
	\end{align*}
	We assume w.l.o.g. that $I, J$ are open.

	Let $x, y \in I$ with $y \geq x$. We calculate
	\begin{align*}
		\phi(y) - \phi(x) 
		&= \left[ \frac{1}{\mu \sqrt{V(s)}} u'(s) \begin{pmatrix}
			- \sin(\mu t(s)) \\ \cos(\mu t(s))
		\end{pmatrix} + u(s) \begin{pmatrix}
			\cos(\mu t(s)) \\ \sin(\mu t(s))
		\end{pmatrix} \right]_x^y
		\\ &= \int_x^y u'(s) \begin{pmatrix}
			- \sin(\mu t(s)) \\ \cos(\mu t(s))
		\end{pmatrix} \der \bigl( \tfrac{1}{\mu \sqrt{V(s)}} \bigr)
		\\ &\qquad+ \int_x^y 
		\frac{u''(s)}{\mu \sqrt{V(s)}} \begin{pmatrix}
			- \sin(\mu t(s)) \\ \cos(\mu t(s))
		\end{pmatrix} + u'(s) \begin{pmatrix}
			- \cos(\mu t(s)) \\ -\sin(\mu t(s))
		\end{pmatrix}\der s
		\\ &\qquad+ \int_x^y u'(s) \begin{pmatrix}
			\cos(\mu t(s)) \\ \sin(\mu t(s))
		\end{pmatrix} + \mu \sqrt{V(s)} u(s) \begin{pmatrix}
			- \sin(\mu t(s)) \\ \cos(\mu t(s))
		\end{pmatrix} \der s
		\\ &= \int_x^y \frac{u'(s)}{\mu} \begin{pmatrix}
			- \sin(\mu t(s)) \\ \cos(\mu t(s))
		\end{pmatrix} \der \bigl( \tfrac{1}{\sqrt{V(s)}} \bigr).
	\end{align*}
	 Applying the triangle inequality to the above integral we obtain
	\begin{align} \label{eq:fuer_gronwall}
		\abs{\phi(y)}_2
		\leq \abs{\phi(x)}_2 
		+ \int_{[x, y]} \frac{\abs{u'(s)}}{\mu} \der \nu
	\end{align}
	where $\nu$ is the total variation of the Lebesgue-Stieltjes measure associated to $\frac{1}{\sqrt{V}}$. Note that $\nu$ is finite with $\nu(I) \leq \Var(\frac{1}{\sqrt{V}}, I)$. Since $u'$ is continuous and
	\begin{align}\label{eq:gronwall_derivative_estimate}
		\frac{\abs{u'(s)}}{\mu} 
		= \sqrt{V(s)} \abs{
			\begin{pmatrix}
				- \sin(\mu t(s)) \\ \cos(\mu t(s))
			\end{pmatrix} \cdot \phi(s)
		}
		\leq \norm{V}_\infty^{\nicefrac12} \abs{\phi(s)}_2
	\end{align}
	holds for $s \in I$, we can apply the Grönwall inequality from \cref{lem:gronwall} to \eqref{eq:fuer_gronwall} and obtain
	\begin{align*}
		\abs{\phi(y)}_2 
		\leq \abs{\phi(x)}_2 \exp \left(\norm{V}_\infty^{\nicefrac12} \nu([x, y])\right)
	\end{align*}
	for $y \geq x$. Setting $E \coloneqq \exp \left(\norm{V}_\infty^{\nicefrac12} \nu(I)\right)$, we have in particular
	\begin{align}\label{eq:phi_sim_const}
		\abs{\phi(y)}_2
		\leq E \abs{\phi(x)}_2
	\end{align}
	for $y \geq x$. A similar argument shows that \eqref{eq:phi_sim_const} also holds for $y \leq x$.
	
	\textit{Part 3:} Let us now estimate the $L^\infty$ and $L^2$-norm. For this, choose some $\xi \in I$. First, we have
	\begin{align*}
		\norm{u}_{L^\infty(I)}
		\leq \max_{x \in I} \abs{\phi(x)}_2
		\leq E \abs{\phi(\xi)}_2.
	\end{align*}
    Now consider the $L^2$-norm. For $\eps > 0$, define an increasing sequence of points $x_n \in \overline{J}$ by 
	$x_0 \coloneqq \inf J$ and $x_{n} \coloneqq \sup\{x \in J \colon \nu((x_{n-1}, x)) \leq \eps\}$ for $n \in \N$. 
	If $x_n < \sup J$ holds for fixed $n \in \N$, we have $\nu((x_{n-1}, x_{n})) \leq \eps \leq \nu((x_{n-1}, x_{n}])$ and therefore
	\begin{align*}
		\nu(I) \geq \sum_{j=1}^n \nu((x_{j-1}, x_j]) \geq n \eps.
	\end{align*}
	Hence the above iteration terminates after $N \leq \ceil{\frac{\nu(I)}{\eps}}$ steps, and thus yields a partition $\inf J = x_0 < x_1 < \dots < x_N = \sup J$ of $J$ with $\nu((x_{n-1}, x_n)) \leq \eps$ for all $n$. For arbitrary $\xi_n \in (x_{n-1}, x_n)$ we calculate
	\begin{align*}
		\MoveEqLeft \int_J \abs{u(s)}^2 \der s
		\\ &= \sum_{n=1}^N \int_{x_{n-1}}^{x_n} \abs{\begin{pmatrix}
			\cos(\mu t(s)) \\ \sin(\mu t(s))
		\end{pmatrix} \cdot \phi(s)}^2 \der s
		\\ &\geq \sum_{n=1}^N \int_{x_{n-1}}^{x_n} \abs{\begin{pmatrix}
			\cos(\mu t(s)) \\ \sin(\mu t(s))
		\end{pmatrix} \cdot \phi(\xi_n)}^2 \der s
		-\sum_{n=1}^N \int_{x_{n-1}}^{x_n} \left(\abs{\phi(s)}_2 + \abs{\phi(\xi_n)}_2\right) \abs{\phi(s) - \phi(\xi_n)}_2 \der s.
	\end{align*}
	We estimate the terms separately. For $n \in \set{1, \dots, N}$, using integration by parts we have
	\begin{align*}
		\MoveEqLeft \int_{x_{n-1}}^{x_n} \abs{\begin{pmatrix}
			\cos(\mu t(s)) \\ \sin(\mu t(s))
		\end{pmatrix} \cdot \phi(\xi_n)}^2 - \tfrac12 \abs{\phi(\xi_n)}_2^2 \der s
		\\ &= \frac{1}{4 \mu} \left[ \frac{1}{\sqrt{V(s)}} \left( \sin(2 \mu t(s)) \left( \abs{\phi_1(\xi_n)}^2 - \abs{\phi_2(\xi_n)}^2\right)+ 4 \sin(\mu t(s))^2 \Re[\phi_1(\xi_n) \overline{\phi_2(\xi_n)}] \right) \right]_{x_{n-1}}^{x_n}
		\\ &\quad- \frac{1}{4 \mu}\int_{x_{n-1}}^{x_n} \left( \sin(2 \mu t(s)) \left( \abs{\phi_1(\xi_n)}^2 - \abs{\phi_2(\xi_n)}^2\right)+ 4 \sin(\mu t(s))^2 \Re[\phi_1(\xi_n) \overline{\phi_2(\xi_n)}] \right) \der \bigl( \tfrac{1}{\sqrt{V(s)}} \bigr).
	\end{align*}
	This allows us to estimate
	\begin{align*}
		\int_{x_{n-1}}^{x_n} \abs{\begin{pmatrix}
			\cos(\mu t(s)) \\ \sin(\mu t(s))
		\end{pmatrix} \cdot \phi(\xi_n)}^2
		&\geq \frac{x_n - x_{n-1}}{2} \abs{\phi(\xi_n)}_2^2
		- \frac{3 \abs{\phi(\xi_n)}_2^2}{4 \mu} \left( 2 \norm{\tfrac{1}{\sqrt{V}}}_\infty +  \nu(J) \right)
		\\ &\geq \frac{x_n - x_{n-1}}{2 E^2} \abs{\phi(\xi)}_2^2
		- \frac{3 E^2 \abs{\phi(\xi)}_2^2}{4 \mu} \left( 2 \norm{\tfrac{1}{\sqrt{V}}}_\infty +  \nu(J)  \right).
	\end{align*}
	For the second set of terms, we use \eqref{eq:gronwall_derivative_estimate} to estimate
	\begin{align*}
		\MoveEqLeft \int_{x_{n-1}}^{x_n} \left(\abs{\phi(s)}_2 + \abs{\phi(\xi_n)}_2\right) \abs{\phi(s) - \phi(\xi_n)}_2 \der s
		\\ &\leq \int_{x_{n-1}}^{x_n} \norm{\phi}_\infty\abs{\int_{\xi_n}^s \frac{u'(s)}{\mu} \begin{pmatrix} - \sin(\mu t(s)) \\ \cos(\mu t(s)) \end{pmatrix} \der \bigl( \tfrac{1}{\sqrt{V(s)}} \bigr)}_2 \der s
		\\ &\leq 2 \abs{\phi(\xi)}_2 (x_n - x_{n-1})  \norm{V}_\infty^{\nicefrac12} \norm{\phi}_\infty^2 \nu((x_{n-1}, x_n))
		\\ &\leq 2 \eps \norm{V}_\infty^{\nicefrac12} E^2 \abs{\phi(\xi)}^2 (x_n - x_{n-1}).
	\end{align*}
	Summing up all estimates over $n$, we get
	\begin{align*}
		\int_I \abs{u(s)}^2 \der s
		\geq \left( \frac{\abs{J}}{2 E^2} - E^2 \left( \frac{3 N}{4 \mu} \left( 2 \norm{\tfrac{1}{\sqrt{V}}}_\infty + \nu(J)\right) + 2 \eps \abs{J} \norm{V}_\infty^{\nicefrac12} \right) \right) \abs{\phi(\xi)}_2^2
	\end{align*}
	The partition of $J$ (and therefore $N$) depends on $\eps$ but not on $\mu = \sqrt{\lambda}$. Therefore, choosing $\eps$ sufficiently small, the constant appearing above is positive for large $\lambda$. This shows
	\begin{align*}
		\norm{u}_{L^\infty(I)} \lesssim \abs{\phi(\xi)}_2 \lesssim \norm{u}_{L^2(I)}.
		&\qedhere
	\end{align*}
	for large $\lambda$, completing the proof.
\end{proof}

\begin{lemma} \label{lem:gronwall}
	Let $I$ be an interval, $f, g \colon I \to [0, \infty)$ be maps, $\nu$ be a locally finite Borel measure on $I$, and $C > 0$. Assume that $g$ is continuous with $g \leq C f$ and
	\begin{align*}
		f(y) \leq f(x) + \int_{[x, y]} g \der \nu
	\end{align*}
	holds for all $y \geq x$. Then we have 
	\begin{align*}
		f(y) \leq f(x) \exp\left(C \nu([x, y])\right)
	\end{align*}
	for all $y \geq x$.
\end{lemma}
\begin{proof}
	Fix $y \geq x$ and $\eps > 0$. As $g$ is uniformly continuous on $[x, y]$, we find $\delta > 0$ such that $\abs{a - b} \leq \delta$ implies $\abs{g(a) - g(b)} \leq \eps$ for $a, b \in [x, y]$. Further choose a partition $x = x_0 < x_1 < \dots < x_n = y$ such that $\abs{x_{m} - x_{m-1}} \leq \delta$ for $m \in \set{1, \dots, n}$ and $\nu(\set{x_m}) = 0$ for $m \in \set{1, \dots, n - 1}$. We then calculate
	\begin{align*}
		f(x_m) 
		&\leq f(x_{m-1}) + \int_{[x_{m-1}, x_m]} g \der \nu
		\\ &\leq f(x_{m-1}) + (g(x_{m-1}) + \eps) \nu([x_{m-1}, x_m])
		\\ &\leq f(x_{m-1}) (1 + C \nu([x_{m-1}, x_m])) + \eps \nu([x_{m-1}, x_m]).
	\end{align*}
	Inserting this inequality into itself for $m = n, \dots, 1$, we obtain
	\begin{align*}
		f(y) 
		&\leq f(x) \prod_{m=1}^{n} (1 + C \nu([x_{m-1}, x_m]))
		+ \eps \sum_{m=1}^n \nu([x_{m-1}, x_m]) \prod_{j=m+1}^{n} (1 + C \nu([x_{j-1}, x_j]))
		\\ &\leq f(x) \exp(C \nu([x, y])) + \eps \nu([x, y]) \exp(C \nu([x, y]))
	\end{align*}
	and the claim follows by letting $\eps \to 0$.
\end{proof}

\begin{remark}
    A similar, more general result, which does not require continuity of $g$ but uses half-open integration intervals can be found in \cite[Theorem~5.1]{groenwall}.
\end{remark}

%% file: examples.tex

\section{Examples} \label{sec:examples}

We analyze the spectrum of $L = - \tfrac{1}{V(x)} \tfrac{d^2}{dx^2}$ in more detail and present examples of $V$ such that our assumptions \ref{ass:first}--\ref{ass:last} hold.
First we observe a qualitative result for the spectrum of $L$. 

\begin{lemma}\label{lem:spec}
    \begin{enumerate}
        \item $\sigma_\ess(L)= \sigma_\ess(- \tfrac{1}{V^-_\per(x)} \tfrac{d^2}{dx^2}) \cup \sigma_\ess( - \tfrac{1}{V^+_\per(x)} \tfrac{d^2}{dx^2})$ 
        \item The spectral bands $\sigma_\ess(L)$ consist of purely absolutely continuous spectrum of $L$ and edges of the spectral bands are no eigenvalues of $L$.
        \item Every gap of $\sigma_\ess(L)$ contains at most finitely many eigenvalues of $L$.
    \end{enumerate}
\end{lemma}
\begin{proof} W.l.o.g. assume $R^- < 0 < R^+$.
    We follow the ideas of \cite{behrndt_smitz_teschl_trunk} and introduce the following notation: Let $L^+_\per$ be a self-adjoint realization of $ - \tfrac{1}{V^+_\per(x)} \tfrac{d^2}{dx^2}$ in $L^2((0,+ \infty); V^+_\per)$ and let $L^-_\per$ be a self-adjoint realization of $ - \tfrac{1}{V^-_\per(x)} \tfrac{d^2}{dx^2}$ in $L^2((- \infty,0); V^-_\per)$.		
    We denote by $L^+$ a self-adjoint realization of $- \tfrac{1}{V(x)} \tfrac{d^2}{dx^2}|_{(0,+\infty)} $ in $L^2((0,+ \infty); V)$  and by $L^-$ a self-adjoint realization of $- \tfrac{1}{V(x)} \tfrac{d^2}{dx^2} |_{(-\infty,0)} $ in $L^2((- \infty,0); V)$.
    
\begin{enumerate}
    \item By Theorem 2.1 in \cite{behrndt_smitz_teschl_trunk} we have $\sigma_\ess( L^\pm_\per)=\sigma_\ess( L^\pm)$ and $L^\pm$ are bounded from below.  
    Next we show that the resolvent difference of $L$ and $L^- \oplus L^+$ is an operator of rank at most two. Since $L^+, L^-$ and $L$ are bounded from below, resolvent operators exist for some constant $\kappa>0$ large enough and we can write
    \begin{align*}
        ((L + \kappa \id)^{-1} - (L^- \oplus L^+ + \kappa \id)^{-1})f=w .
    \end{align*}
    We set $(L + \kappa \id)u=f$ and $(L^- \oplus L^+ + \kappa \id)v=f$ and hence $w=u-v$.
    For the restriction on $(0,\infty)$ we have $(L + \kappa \id) w|_{(0,\infty)}=0$, therefore $w |_{(0,\infty)}$ lies in the kernel of $(L + \kappa \id)|_{(0,\infty)}$ which has dimension at most one since $L+\kappa I$ is of limit-point type at $+\infty$, cf. Theorem~\ref{thm:limit_point}. The same holds true for the restriction on $(-\infty,0)$.
    By Corollary~11.2.3 in \cite{davies} we obtain $\sigma_\ess(L)=\sigma_\ess(L^- \oplus L^+) = \sigma_\ess(L^-_\per) \cup \sigma_\ess( L^+_\per)$.

	\item See the proof of Theorems~1.1-1-3 in \cite{behrndt_smitz_teschl_trunk}.

	\item By Theorem 2.3 from \cite{behrndt_smitz_teschl_trunk} we know that every gap of $\sigma_\ess(L^\pm)$ contains at most finitely many eigenvalues of $L^\pm$ and therefore also every gap of $\sigma_\ess(L^- \oplus L^+)$ contains at most finitely many eigenvalues of $L^- \oplus L^+$. Since the resolvent difference of $L$ and $L^- \oplus L^+$ is an operator of rank at most two we conclude by Theorem 3, Chapter 9.3 in \cite{birman_solomjak} that in each gap of $\sigma_\ess(L)$ the operator $L$ gains at most two more eigenvalues compared to the finitely many eigenvalues of $L^-\oplus L^+$ in each gap of $\sigma_\ess(L^-\oplus L^+)$.  
	\qedhere
\end{enumerate}
\end{proof}

\subsection{Purely periodic case}
As potential $V$ we consider a positive periodic step function $V_\per$ given as follows:
take a partition $0=\theta_0<\theta_1<\ldots<\theta_N=1$ of the interval $[0,1]$ and positive values $a_1,\ldots, a_N>0$ to define 
\begin{align}\label{eq:multistepV}
    V_\per(x) = a_i \mbox{ for } x\in [\theta_{i-1}X, \theta_i X) \mbox{ and } i=1,\ldots, N
\end{align}
and extend $V_\per$ periodically to the real line with period $X$. First, we note that our assumptions \ref{ass:bounded} and \ref{ass:periodic_infinity} are satisfied by definition of $V_\per$. Moreover, by Theorem 5.3.1 in \cite{Eastham}, $\sigma_p(L)=\emptyset$ and hence \ref{ass:point_spectrum} is fulfilled.

\begin{lemma}\label{lem:multistepcond} Let $q_i \coloneqq \sqrt{a_i}(\theta_{i}-\theta_{i-1})X$ for $i=1,\ldots, N$.
    Assume that there is $T>0$ such that 
    $$
    4q_i \in T\N, \quad i=1,\ldots,N
    $$ 
    and suppose that 
    $$
    4 q_{i_j}\in T\Nodd
    $$
    is satisfied for an even number of indices $1\leq i_1<i_2<\ldots<i_{2m}\leq N$ and no others. If moreover
    \begin{align} \label{eq:quotient}
    \alpha := \frac{a_{i_1} a_{i_3}\cdot \ldots \cdot a_{i_{2m-1}}}{a_{i_2} a_{i_4} \cdot \ldots \cdot a_{i_{2m}}} \not =1
   \end{align}
    then \ref{ass:spectrum} is satisfied for $\omega = \frac{2 \pi}{T}$ and $V = V_\per$ from above.
\end{lemma}
\begin{proof} We denote by $P_i(\lambda)$ the weighted monodromy matrix such that any solution of $-u'' = \lambda a_i u$ on $[\theta_{i-1}X, \theta_i X]$ satisfies
\begin{align} \label{eq:prop_matrix}
     \begin{pmatrix}
        \sqrt{\lambda}u(\theta_i X) \\ u'(\theta_i X)
    \end{pmatrix} = P_i(\lambda) \begin{pmatrix}
        \sqrt{\lambda} u(\theta_{i-1}X) \\ u'(\theta_{i-1}X) 
    \end{pmatrix} .
\end{align}
It is given by
\begin{align} \label{eq:propmatrix_2}
    P_i(\lambda)= \begin{pmatrix}
        \cos(\sqrt{\lambda} q_i) & \frac{1}{\sqrt{a_i}}\sin(\sqrt{\lambda} q_i) \\
         -\sqrt{a_i}\sin(\sqrt{\lambda} q_i) & \cos(\sqrt{\lambda} q_i)
    \end{pmatrix}
\end{align}
and as a function of $\sqrt{\lambda}$ it is $\frac{2\pi}{q_i}$-periodic. Since $q_i\in \frac{T}{4}\N$ it is in particular also $\frac{8\pi}{T}$-periodic. The weighted monodromy matrix $P(\lambda;V_\per)$ of the full problem $-u''= \lambda V_\per(x) u$ on $[0,X]$ is then given by 
$$
P(\lambda;V_\per) = P_{N}(\lambda)\cdot \ldots \cdot P_1(\lambda),
$$
which is $4\omega=\frac{8\pi}{T}$-periodic as a function of $\sqrt{\lambda}$. 
Following Chapter 1 and Section 2.1 in \cite{Eastham} we know that 
\begin{align*}
    \sigma\left(-\frac{1}{V_\per(x)} \dxsquare \right)= \left\{ \lambda \in \R \colon |\tr  P(\lambda;V_\per) |\leq 2 \right\}
\end{align*}
where we use that the weighted monodromy matrix $P(\lambda;V_\per)$ and the standard monodromy matrix are similar and therefore have the same eigenvalues and the same trace. In order to check that \ref{ass:spectrum} holds, the $4\omega$-periodicity w.r.t. $\sqrt{\lambda}$ implies that it suffices to check that $|\tr P(\omega^2;V_\per)|, |\tr P(9\omega^2;V_\per)|>2$ since $\omega, 3\omega$ are the only odd multiples of $\omega$ in the periodicity cell $[0,4\omega]$. If we insert $k^2\omega^2$, $k\in \Nodd$ into $P_i(\lambda)$ then the assumptions yield that 
$$
    P_i(k^2\omega^2)= \pm \begin{pmatrix}
        0 & \frac{1}{\sqrt{a_i}}\\
         -\sqrt{a_i} & 0
    \end{pmatrix} \mbox{ if } 4q_i \in T\Nodd\quad \mbox{ and } \quad P_i(k^2\omega^2)= \pm\Id \mbox{ if } 4q_i \in T\Neven.
$$
By definition of the indices $i_1, \dots, i_{2m}$ we have
$$
P_{i_{j+1}}(k^2\omega^2) P_{i_j}(k^2\omega^2)= \pm \begin{pmatrix} \sqrt{\frac{a_{i_j}}{a_{i_{j+1}}}} & 0 \\ 0 & \sqrt{\frac{a_{i_{j+1}}}{a_{i_j}}}\end{pmatrix}
$$
and obtain 
\begin{align} \label{eq:structure_prop_per}
P(k^2\omega^2;V_\per) = \pm \begin{pmatrix}
        \sqrt{\alpha} & 0\\
         0 & \frac{1}{\sqrt{\alpha}}
    \end{pmatrix} \quad 
\end{align}
so that $|\tr{P(k^2\omega^2;V_\per)}| >2$ by \eqref{eq:quotient}. Therefore assumption~\ref{ass:spectrum} holds. 
\end{proof}

\begin{lemma} \label{lem:sufficient}
	Let $q_i \coloneqq \sqrt{a_i} (\theta_i - \theta_{i-1}) X$ for $i = 1, \dots, N$ be pairwise rational multiples of one another with 
    greatest common divisor $q \coloneqq \gcd(q_1, \dots, q_N)$ defined as the largest positive number such that all $q_i$ are integer multiples of $q$. If $\frac{q_i}{q}\in \Nodd$ is satisfied for an even number of indices $1\leq i_1<i_2<\ldots<i_{2m}\leq N$ and no others, and additionally 
    \begin{align*}
		\frac{a_{i_1} a_{i_3}\cdot \ldots \cdot a_{i_{2m-1}}}{a_{i_2} a_{i_4} \cdot \ldots \cdot a_{i_{2m}}} \neq 1
	\end{align*}
	holds then the assumptions of Lemma~\ref{lem:multistepcond} hold with $T = 4 \frac{q}{k}$ for any $k \in \Nodd$.
\end{lemma}

\begin{remark}
    One can check that the conditions of Lemma~\ref{lem:sufficient} are not only sufficient but also necessary for Lemma~\ref{lem:multistepcond}.
\end{remark}

\begin{proof}[Proof of \cref{lem:sufficient}]
	The condition $4 q_i \in T \N$ for all $\N$ means that $\frac{T}{4}$ is a common divisor of the $q_i$, and is therefore equivalent to $\frac{T}{4} = \frac{q}{k}$ for some $k \in \N$. We now have to check condition \eqref{eq:quotient}, where we have
	\begin{align*}
		4 q_i \in T \Nodd
		\iff k \frac{q_i}{q} \in \Nodd 
		\iff k \in \Nodd \text{ and } \frac{q_i}{q} \in \Nodd.
	\end{align*}
	Thus, for even $k$ we have $m = 0$ in \cref{lem:multistepcond} so that \eqref{eq:quotient} is false, whereas for odd $k$ the indices $i_1, \dots, i_{2m}$ of \cref{lem:multistepcond} coincide with the indices of the current lemma, so that \eqref{eq:quotient} holds by assumption.
\end{proof}

\begin{remark}[Two-step and three-step potentials]
	Let $N=2$ and assume that $a_1, a_2>0$, $0=\theta_0<\theta_1<\theta_2=1$ satisfy the assumptions from Lemma~\ref{lem:sufficient}. This means that $\frac{q_2}{q_1}= \frac{r}{s}$ with $r,s\in \Nodd$ coprime and $T\in \frac{4q}{\Nodd}$ with the greatest common divisor $q= \frac{q_1}{s}=\frac{q_2}{r}$. These conditions coincide with those stated in \cite{kohler_reichel} on admissible ranges of breather frequencies and material parameters for the case $X=2 \pi$. 
	
	In the case $N=3$, assuming $a_1,a_2,a_3 >0$ and $0=\theta_0 < \theta_1 < \theta_2 < \theta_3=1$ satisfying the conditions from Lemma \ref{lem:sufficient} means (up to a permutation of the $q_i$) that $\frac{q_2}{q_1}=\frac{r}{s}$ with $r,s \in \Nodd$ coprime, $\frac{q_3}{q_1}=\frac{\tilde{r}}{\tilde{s}}$ with $\tilde{r} \in \Neven$, $\tilde{s} \in \Nodd$ coprime and $a_1 \neq a_2$. Further, the greatest common divisor is given by $q=\gcd(q_1,q_2,q_3)=\frac{q_1}{\lcm(s, \tilde{s})}$ where $\lcm$ denotes the least common multiple.
\end{remark}

\subsection{Perturbed periodic case}
Next, we want to analyze the spectrum of $L$ when $V$ is a perturbed periodic potential and focus on two different cases. 

\subsubsection{Interface of dislocated periodic potentials}

Let $V_\per$ be given by \eqref{eq:multistepV} and for $V_0, d >0$ consider 
\begin{align*}
    V(x)=\begin{cases}
        V_\per(x),  & x <0,\\
        V_0,  & 0 \leq x < d, \\
        V_\per(x-d),  & d \leq x . \\
    \end{cases}
\end{align*}
We observe that our assumptions \ref{ass:bounded} and \ref{ass:periodic_infinity} hold by definition.

\begin{lemma}
    Assume that $V_\per$ satisfies the assumptions of Lemma~\ref{lem:multistepcond} and that for the same value of $T$ we have $4q_0\in T\Neven$ with $q_0:= \sqrt{V_0}d$. Then \ref{ass:spectrum} and \ref{ass:point_spectrum} hold. 
\end{lemma}

We state the proof together with the proof of Lemma \ref{lem:interface} since they are based on the same idea.

\subsubsection{Interface of periodic potentials}

Let $V_\per^+$ and $V_\per^-$ be different periodic potentials given by \eqref{eq:multistepV}.
Consider now the case
\begin{align*}
    V(x) = \begin{cases}
        V^-_\per(x) , & x<0 , \\
        V^+_\per(x) , & x\geq 0 .
    \end{cases}
\end{align*}
Note that assumptions \ref{ass:bounded} and \ref{ass:periodic_infinity} hold by definition.

\begin{lemma}\label{lem:interface} 
    Assume that $V^\pm_\per$ both satisfy the assumptions of Lemma~\ref{lem:multistepcond} for the same value of $T$ with values $\alpha^\pm$ from \eqref{eq:quotient}. If additionally $\alpha^+, \alpha^- > 1$ or $\alpha^+, \alpha^- < 1$ then \ref{ass:spectrum} and \ref{ass:point_spectrum} hold.
\end{lemma}

\begin{proof}
Observe first that according to Lemma~\ref{lem:spec} and Lemma~\ref{lem:multistepcond} we know that \ref{ass:spectrum} holds for the essential spectrum of $L$ and hence it remains to verify \ref{ass:spectrum} and \ref{ass:point_spectrum} for the eigenvalues of
\begin{align*}
    -u''= \lambda V(x) u \text{ for } x \in \R .
\end{align*}

In the following we use from Lemma~\ref{lem:multistepcond} that the square root of the spectrum of $L_\per$ is periodic with period $4\omega=\frac{8\pi}{T}$.

\emph{Interface of dislocated periodic potentials.}
In analogy to \eqref{eq:prop_matrix} we define the weighted propagation matrix $P_0(\lambda)$ for solutions of  $ - u''= \lambda V_0 u $ on $[0,d]$. It takes the form as in \eqref{eq:propmatrix_2} and hence, as a function of $\sqrt{\lambda}$, it is periodic with period $\frac{2\pi}{q_0}$. By the assumption $4q_0\in T\Neven$ it is co-periodic to the propagation matrix $P(\lambda;V_\per)$ of the periodic potential $V_\per$, which has period $4\omega=\frac{8\pi}{T}$.  

Let us consider a value $\lambda \in \sigma_p(L)$. 
By \cref{lem:spec} we have $\lambda\not \in \sigma_\ess(L)$ and hence $|\tr P(\lambda;V_\per)|>2$. Then, $P(\lambda;V_\per)$ has two distinct real eigenvalues $\rho(\lambda), \tilde \rho(\lambda)$ with $|\rho(\lambda)|<1<|\tilde \rho(\lambda)|$ and corresponding eigenvectors $v(\lambda), \tilde v(\lambda)\in \R^2$. 
We note that $\rho(\lambda), \tilde\rho(\lambda), \R v(\lambda), \R\tilde v(\lambda)$ are $4\omega$-periodic as functions of $\sqrt{\lambda}$ (inside the resolvent set of $L_\per$). 
If $\lambda \in \sigma_p(L)$ is an eigenvalue with $L^2(\R)$-eigenfunction $\phi$, we necessarily have $(\sqrt{\lambda} \phi(0), \phi'(0)) \in \R \tilde v(\lambda)$ and $(\sqrt{\lambda} \phi(d), \phi'(d)) \in \R v(\lambda)$ and therefore $P_0(\lambda) \tilde v(\lambda) \in \R v(\lambda)$.
Since this condition is also $4\omega$-periodic as a function of $\sqrt{\lambda}$ we see that $\sigma_p(L)$ has the same finite number of eigenvalues in every interval of length $4\omega$. This already implies \ref{ass:point_spectrum}. 
For \ref{ass:spectrum} we only need to check that $k^2\omega^2\not\in \sigma_p(L)$ for $k\in \Nodd$. From \eqref{eq:structure_prop_per} we see by the structure of the propagation matrices that 
$$
\alpha>1 \implies v(k^2\omega^2)= \begin{pmatrix}
    0 \\1 
\end{pmatrix}, \quad \tilde v(k^2\omega^2)= \begin{pmatrix}
    1 \\0 
\end{pmatrix} 
$$
and 
$$
\alpha<1 \implies v(k^2\omega^2)= \begin{pmatrix}
    1 \\0 
\end{pmatrix}, \quad \tilde v(k^2\omega^2)= \begin{pmatrix}
    0 \\1 
\end{pmatrix} 
$$
up to rescaling of the eigenvectors.
Since $P_0(k^2\omega^2)=\pm \Id$, we get $k^2\omega^2 \not \in \sigma_p(L)$ and hence also \ref{ass:spectrum} holds.

\emph{Interface of periodic potentials.} The considerations are similar to the previous case. As before, we denote by $P(\lambda;V^\pm_\per)$ the propagation matrix for $V^\pm_\per$. By our assumption, both have the same period $4\omega= \frac{8\pi}{T}$. For $\lambda \not \in \sigma_\ess(L)$, the matrix $P(\lambda; V^\pm_\per)$ has eigenvalues $\rho^\pm(\lambda), \tilde \rho^\pm(\lambda)$ with eigenvectors $v^\pm(\lambda), \tilde v^\pm(\lambda)$.  The eigenpair $(\rho^+(\lambda), v^+(\lambda))$ generates a solution on $[0,\infty)$ decaying to $0$ at $+\infty$ and $(\tilde \rho^-(\lambda),\tilde v^-(\lambda))$ generates a solution on $(-\infty,0]$ decaying to $0$ at $-\infty$. Therefore, the eigenvalue condition is given by 
$$
\lambda\in \sigma_p(L)
\iff
\tilde v^-(\lambda) \in \R v^+(\lambda).
$$
Recall that at $\lambda=k^2\omega^2$ we have 
$$
\alpha^+>1 \implies v^+(k^2\omega^2)= \begin{pmatrix}
    0 \\1 
\end{pmatrix}, \quad \alpha^+<1 \implies v^+(k^2\omega^2)= \begin{pmatrix}
    1 \\0 
\end{pmatrix} 
$$
and 
$$
\alpha^->1 \implies \tilde v^-(k^2\omega^2)= \begin{pmatrix}
    1 \\0 
\end{pmatrix}, \quad \alpha^-<1 \implies \tilde v^-(k^2\omega^2)= \begin{pmatrix}
    0 \\1 
\end{pmatrix}. 
$$
Using our assumption that $\alpha^+, \alpha^- > 1$ or $\alpha^+, \alpha^- < 1$ we conclude $k^2\omega^2\not\in \sigma_p(L)$ and hence \ref{ass:spectrum} and \ref{ass:point_spectrum} holds as seen in the previous case.
\end{proof}